\documentclass[10pt]{article}

\usepackage[english]{babel}

\usepackage{amsmath}
\usepackage{amsthm}
\usepackage{tikz}
\usepackage{amssymb}
\usepackage{mathtools}
\usepackage{graphicx}
\usepackage{xcolor}
\usepackage[hidelinks, bookmarks = false]{hyperref}
\usepackage{enumitem}
\usepackage{bbm}

\usepackage[square,sort,comma,numbers]{natbib}

\newtheorem{theorem}{Theorem}[section]

\newtheorem{definition}[theorem]{Definition}
\newtheorem{conjecture}[theorem]{Conjecture}
\newtheorem{question}[theorem]{Question}
\newtheorem{lemma}[theorem]{Lemma}
\newtheorem{proposition}[theorem]{Proposition}

\makeatletter
\newcommand*{\rom}[1]{\expandafter\@slowromancap\romannumeral #1@}
\makeatother

\def\ex{\mathrm{ex}}
\def\Aut{\mathrm{Aut}}

\def\QED{{\hfill\enspace\vrule height8pt depth0pt width8pt}}

\begin{document}

\title{Supersaturation beyond color-critical graphs}
\author{
Jie Ma$^{1,3}$\and
\and
Long-Tu Yuan$^2$
}
\date{}
\maketitle

\footnotetext[1]{School of Mathematical Sciences, University of Science and Technology of China, Hefei, Anhui, 230026, China. Research supported by National Key Research and Development Program of China 2020YFA0713100, National Natural Science Foundation of China grant 12125106, and Anhui Initiative in Quantum Information Technologies grant AHY150200. Email: jiema@ustc.edu.cn.}
\footnotetext[2]{School of Mathematical Sciences, Key Laboratory of MEA(Ministry of Education) \& Shanghai
 Key Laboratory of PMMP, East China Normal University, Shanghai 200240, China.
Research supported by National Natural Science Foundation of China grant 12271169, 12331014 and Science and Technology Commission of Shanghai Municipality 22DZ2229014.
Email: ltyuan@math.ecnu.edu.cn.}
\footnotetext[3]{Hefei National Laboratory, University of Science and Technology of China, Hefei 230088, China. Research supported by Innovation Program for Quantum Science and Technology 2021ZD0302902.}

\begin{abstract}
The supersaturation problem for a given graph $F$ asks for the minimum number $h_F(n,q)$ of copies of $F$ in an $n$-vertex graph with $\ex(n,F)+q$ edges.
Subsequent works by Rademacher, Erd\H{o}s, and Lov\'{a}sz and Simonovits determine the optimal range of $q$ (which is linear in $n$) for cliques $F$ such that
$h_F(n,q)$ equals the minimum number $t_F(n,q)$ of copies of $F$ obtained from a maximum $F$-free $n$-vertex graph by adding $q$ new edges.
A breakthrough result of Mubayi extends this line of research from cliques to color-critical graphs $F$,
and this was further strengthened by Pikhurko and Yilma who established the equality $h_F(n,q)=t_F(n,q)$ for $1\leq q\leq \epsilon_F n$ and sufficiently large $n$.
In this paper, we present several results on the supersaturation problem that extend beyond the existing framework.
Firstly, we explicitly construct infinitely many graphs $F$ with restricted properties for which $h_F(n,q)<q\cdot t_F(n,1)$ holds when $n\gg q\geq 4$, thus refuting a conjecture of Mubayi.
Secondly, we extend the result of Pikhurko-Yilma by showing the equality $h_F(n,q)=t_F(n,q)$ in the range $1\leq q\leq \epsilon_F n$ for any member $F$ in a diverse and abundant graph family (which includes color-critical graphs, disjoint unions of cliques $K_r$, and the Petersen graph).
Lastly, we prove the existence of a graph $F$ for any positive integer $s$ such that $h_F(n,q)=t_F(n,q)$ holds when $1\leq q\leq \epsilon_F n^{1-1/s}$, and $h_F(n,q)<t_F(n,q)$ when $n^{1-1/s}/\epsilon_F\leq q\leq \epsilon_F n$,
indicating that $q=\Theta(n^{1-1/s})$ serves as the threshold for the equality $h_F(n,q)=t_F(n,q)$.
We also discuss some additional remarks and related open problems.
\end{abstract}

\section{Introduction}
Let $F$ be a graph. A graph is {\it $F$-free} if it does not contain $F$ as a subgraph.
The {\it Tur\'{a}n number} $\ex(n,F)$ of $F$ denotes the maximum number of edges in an $n$-vertex $F$-free graph.
An $n$-vertex graph is called an {\it extremal graph for $F$} if it is $F$-free and has the maximum number $\ex(n,F)$ of edges.
In this paper, we study the {\it supersaturation problem for $F$}, that is, to determine the minimum number $h_F(n,q)$ of copies of $F$ in an $n$-vertex graph with $\ex(n,F)+q$ edges.
A related concept is the minimum number $t_F(n,q)$ of copies of $F$ in graphs obtained from an $n$-vertex extremal graph for $F$ by adding $q$ new edges.
It is worth noting that $h_F(n,q) \leq t_F(n,q)$, and extensive research has been conducted in the literature to establish the equality $h_F(n,q) = t_F(n,q)$ under certain circumstances.
This paper presents results on the supersaturation problem that go beyond the existing framework,
showcasing intricate and unexpected relations between $h_F(n,q)$, $q\cdot t_F(n,1)$, and $t_F(n,q)$ in particular.

The celebrated Tur\'{a}n theorem \cite{turan1941} (the case $r=2$ was first proved by Mantel \cite{Man07}) states that any $n$-vertex graph with $t_r(n)+1$ edges contains at least one copy of $K_{r+1}$, where $t_r(n)$ denotes the number of edges in the Tur\'an graph $T_r(n)$, i.e., the complete $r$-partite $n$-vertex graph.
In 1941, Rademacher proved that any $n$-vertex graph with $t_2(n)+1$ edges contains at least $\lfloor n/2\rfloor$ copies of $K_3$.
Stated in the above context, we have the equality $h_{K_3}(n,1)=\lfloor n/2\rfloor= t_{K_3}(n,1)$.
This result is often recognized as the starting point for the study on the supersaturation problem in extremal graph theory.
In subsequent papers \cite{erd55,erd62-1}, Erd\H{o}s extended this by showing that: there exists a constant $\epsilon_3>0$ so that
$$\mbox{$h_{K_3}(n,q)=t_{K_3}(n,q)$ holds for any $1\leq q<\epsilon_{3}n$.}$$
Later Lov\'{a}sz and Simonovits \cite{Lovasz1976} determined the optimal value of $\epsilon_{3}$ as $n\to \infty$, confirming a longstanding conjecture of Erd\H{o}s.
In a subsequent work, Lov\'{a}sz and Simonovits \cite{Lovasz1983} extended their result from the triangle $K_3$ to every clique $K_r$,
establishing the equality $h_{K_r}(n,q)=t_{K_r}(n,q)$ for any $1\leq q<\epsilon_{r}n$ with the best constant $\epsilon_r$.
In fact Lov\'{a}sz and Simonovits \cite{Lovasz1983} completely solved the supersaturation problem for cliques $K_r$ with $r\geq 3$ when $q=o(n^2)$.
The case $q=\Omega(n^2)$ of the supersaturation problem for cliques $K_r$ has also been extensively studied, see \cite{Fisher1989,Fisher1992,razborov2008,nikiforov2011,Reiher2016,LPS2020} and the references therein.

The supersaturation problems were also investigated for general graphs beyond just cliques.
For bipartite graphs,
the captivating conjecture put forth by Erd\H{o}s-Simonovits \cite{Simonovits1984} and Sidorenko \cite{Sidorenko1993} has received significant attention and extensive research efforts. However, in the scope of this paper, we will not delve into a detailed discussion of this conjecture, and instead, we will focus on non-bipartite graphs.
Now let $F$ be a non-bipartite non-clique graph.
By the number of copies of $F$ in a given graph $G$, we mean the number of edge subsets $A\subseteq E(G)$ which induces an copy of $F$.
This also equals the number of edge-preserving injections from $V(F)$ to $V(G)$ divided by $\Aut(F)$, where $\Aut(F)$ denotes the number of automorphisms of $F$.
A graph is {\it color-critical} if it contains an edge whose deletion reduces its chromatic number.
The family of color-critical graphs plays an important role in the development of extremal graph theory.
A classic theorem of Simonovits \cite{Simonovits1968} states that the Tur\'an graph $T_r(n)$ is the unique extremal graph for any color-critical graph $F$ with chromatic number $r+1$ when $n$ is sufficiently large.
In other words, he proved that if $n$ is sufficiently large then any $n$-vertex graph with $t_r(n)+1$ edges contains at least one copy of such $F$.
In a breakthrough paper, Mubayi \cite{mubayi2010} extended Simonovits' theorem using a novel and unified approach for color-critical graphs.
Throughout this paper, for any graph $F$, let $c(n,F)$ be the minimum number of copies of $F$ obtained from an $n$-vertex extremal graph for $F$ by adding one edge.\footnote{Note that for any graph $F$, we have $c(n,F)=t_F(n,1)$ and $t_F(n,q)\geq q\cdot c(n,F)$ for any $q\geq 1$.}
\begin{theorem}[Mubayi \cite{mubayi2010}]\label{Thm:Mubayi}
For every color-critical graph $F$ with chromatic number $r+1$, there exists a constant $\delta=\delta_F>0$ such that if $n$ is sufficiently large and $1\leq q\leq \delta n$, then any $n$-vertex graph with $t_r(n)+q$ edges contains at least $q\cdot c(n,F)$ copies of $F$. That is, $h_F(n,q)\geq q\cdot c(n,F).$
\end{theorem}
One significant aspect of this result is its utilization of the Graph Removal Lemma (see e.g. \cite{Komlos1996}) and the Erdos-Simonovits Stability Theorem \cite{erdHos1967,erdHos1968,Simonovits1968} to accurately count substructures in graphs.
We point out that provided $1\leq q\leq \delta n$, the lower bound $h_F(n,q)\geq q\cdot c(n,F)$ is sharp for many color-critical graphs $F$ (including cliques, odd cycles, and the graph obtained from $K_4$ by deleting an edge); moreover, it is asymptotically tight for any color-critical graph $F$ due to the following fact:
$$q\cdot c(n,F)\leq t_F(n,q)\leq (1+o(1))q\cdot c(n,F) ~~~ \Longrightarrow ~~~ (1-o(1))t_F(n,q)\leq h_F(n,q)\leq t_F(n,q) \mbox{ for } 1\leq q\leq \delta n.$$

This line of research on color-critical graphs was further enhanced by Pikhurko and Yilma \cite{pikhurko2017}.
Among other results, they proved the following strengthening of Theorem~\ref{Thm:Mubayi}.
\begin{theorem}[Pikhurko and Yilma \cite{pikhurko2017}]\label{Thm:PY}
For every color-critical graph $F$, there exists a constant $\delta=\delta_F>0$ such that if $n$ is sufficiently large $n$ and $1\leq q\leq \delta n$, then $h_F(n,q)=t_F(n,q).$
\end{theorem}
The authors \cite{pikhurko2017} also determined $h_F(n,q)$ asymptotically for any color-critical graph $F$ in the case $q=o(n^2)$,
by reducing to some optimization problems (see Theorems 3.10-3.11 in \cite{pikhurko2017}).
Of particular interest to them is identifying a threshold for when graphs obtained from extremal graphs for $F$ by adding $q$ new edges
are optimal or asymptotically optimal in the range $q=O(n)$ (e.g., equations (3) and (4) in \cite{pikhurko2017}).
We will explore this intriguing question, showing that such thresholds can be rather sophisticated.

To the best of our knowledge, the study of supersaturation problems for non-bipartite graphs, specifically excluding color-critical cases, has only recently been undertaken for the ``bowtie" graph,
which consists of two copies of $K_3$ merged at a vertex, as explored by Kang, Makai and Pikhurko in \cite{Kang2017}.
On the other hand, the powerful approach utilizing the graph removal lemma and the Erdos-Simonovits stability theorem, as introduced in \cite{mubayi2010}, was effectively employed in the proof of the aforementioned Theorem~\ref{Thm:PY} of \cite{pikhurko2017}, and subsequently extended to hypergraph settings in \cite{mubayi2013, mubayi2013-2}.
These results ``suggest that whenever one can obtain stability and exact results for an extremal problem, one can also obtain counting results'', cited from \cite{mubayi2013}.
In an effort to unify this approach, Mubayi \cite{mubayi2013} formulated a conjecture as follows.
An $r$-uniform hypergraph (i.e., an {\it $r$-graph} in short) $F$ is {\bf stable} if $\ex(n,F)$ is achieved by a unique $n$-vertex $r$-graph $H(n)$ for sufficiently large $n$, and every $n$-vertex $F$-free $r$-graph with $(1-o(1))\ex(n,F)$ edges can be obtained from $H(n)$ by changing at most $o(n^r)$ edges.

\begin{conjecture}[Mubayi, Conjecture 5.1 in \cite{mubayi2013}]\label{con-mubayi}
Let $r\geq 2$ and let $F$ be a non $r$-partite stable $r$-graph.
For every positive integer $q$, if $n$ is sufficiently large, then
$h_F(n,q)\geq q\cdot c(n,F)$.\footnote{Here, these definitions for $r$-graphs $F$ are analogously defined.}
\end{conjecture}

In this paper, we investigate supersaturation problems beyond color-critical graphs while exploring the corresponding natural enumerative parameters.
Our first result refutes Conjecture~\ref{con-mubayi} in the graph case by providing a counterexample for every integer $q\geq 4$, in the following strong form.

\begin{theorem}\label{main-example}
There exists a non-bipartite stable graph $F$ such that the following holds.
There exist a small constant $\delta=\delta_F>0$ and an integer $n_0=n_0(F)$ such that for any integers $n \geq n_0$ and $4 \leq q \leq \delta n$,
it holds that $$\frac{h_F(n,q)}{q \cdot c(n,F)}\leq 1-\delta.$$
\end{theorem}
\noindent The proof of this result actually yields infinitely many counterexamples $F$ with arbitrary chromatic number at least four to Conjecture~\ref{con-mubayi}. Additionally, since $t_F(n,q)\geq q\cdot c(n,F)$, this implies that for such $F$,
$$\mbox{$h_F(n,q)<t_F(n,q)$ holds for any fixed $q\geq 4$ and sufficiently large $n$.}$$
To the best of our knowledge, these examples represent the first instances with the above property for general graphs.
We will discuss more about related problems in the concluding remarks.

Our second main result extends Theorem~\ref{Thm:PY} to a diverse and abundant family of graphs.
The precise definition of this family requires the introduction of some technical notations, which we will defer until Definition~\ref{def 3}.
We mention here that this family includes color-critical graphs, Kneser graphs $K(t,2)$, disjoint unions of cliques $K_r$, and many others (see the remarks following Definition~\ref{def 3}).
In the subsequent statement, we focus solely on the Kneser graphs $K(t,2)$,
which are the graphs with the vertex set $\binom{[t]}{2}$ where two vertices $A$ and $B$ in $\binom{[t]}{2}$ are adjacent if and only if $A\cap B=\emptyset$;
we refer to Subsection~\ref{Subs:Kneser} for a detailed discussion on extremal results concerning the Kneser graphs $K(t,2)$.

\begin{theorem}\label{Thm:Knes}
For any Kneser graph $K=K(t,2)$ with $t\geq 5$, there exists a constant $\delta>0$ such that for any sufficiently large integer $n$ and any integer $1\leq q\leq \delta n$, we have $h_K(n,q)=t_K(n,q).$
\end{theorem}

\noindent A notable case is the Petersen graph $\mathbf{P}$, which corresponds to the Kneser graph $K(5,2)$.
As a prompt corollary, one can deduce from Theorem~\ref{Thm:Knes} and an old result of Simonovits on $\ex(n,\mathbf{P})$ \cite{Simonovits1974} that for sufficiently large $n$,
$$h_{\mathbf{P}}(n,1)=c(n,\mathbf{P})=96{\lceil\frac{n}{2}\rceil-3 \choose 2}{ \lfloor\frac{n}{2}\rfloor-1 \choose 4}\approx \frac{n^6}{32}.$$
Our proof, similar to \cite{mubayi2010,pikhurko2017}, employs the graph removal lemma and the Erdos-Simonovits stability theorem as the main tools, while also requiring novel techniques for counting substructures in various scenarios.
The full statement of this result can be found in Theorem~\ref{Thm:adm}.

Our final result explores the thresholds for the equality $h_F(n,q)=t_F(n,q)$ to hold as $q$ varies as a function of $n$ for graphs $F$.
As noted previously, this question was examined in \cite{pikhurko2017} for color-critical graphs.
The following result indicates that for any positive integer $s$, this threshold can be achieved with $q=\Theta(n^{1-1/s})$ for some non-bipartite stable graph $F$.

\begin{theorem}\label{Thm:hVSt}
For any positive integer $s$,
there exists a non-bipartite stable graph $F$ such that the following holds.
There is a constant $\epsilon>0$ such that for every sufficiently large integer $n$,
\begin{itemize}
\item[(1)] if $1\leq q \le \epsilon n^{1-1/s}$, then $h_F(n,q)=t_F(n,q)$, and
\item[(2)] if $n^{1-1/s}/\epsilon \le q\le \epsilon n$, then $h_F(n,q)<t_F(n,q)$.
\end{itemize}
\end{theorem}

The organization of this paper is as follows:
In Section~\ref{Sec:pre}, we provide preliminaries, including notations, key lemmas, and the definition of a graph family that plays a crucial role throughout this paper.
Section~\ref{Sec:example} presents an explicit example to prove Theorem~\ref{main-example} and refute Conjecture~\ref{con-mubayi}.
In Section~\ref{sec_properties-of-H}, we establish quantitative and structural properties for graphs with the minimum number of copies of $F$, which are essential for the subsequent sections.
Section~\ref{Section:proof of theorem 1.6} introduces a special family of graphs and demonstrates that for any graph $F$ in this family, the equality $h_F(n,q)=t_F(n,q)$ holds for $1\leq q\leq \epsilon_F n$ and sufficiently large $n$, implying Theorem~\ref{Thm:Knes}.
In Section~\ref{section s and t}, we complete the proof of Theorem~\ref{Thm:hVSt}.
Finally, in the concluding section, we provide several remarks and discuss related problems.

\section{Preliminaries}\label{Sec:pre}
\subsection{Notations}
Let $G$ be a given graph.
The {\it neighborhood} of a vertex $u$ in $G$ is denoted by $N_G(u)=\{v\in V(G): uv\in E(G)\}$.
By $N_G[u]$ we denote the set $N_G(u)\cup\{u\}$.
The {\it degree} $d_{G}(u)$ of the vertex $u$ in $G$ is the size of $N_G(u)$.
For an edge subset $A\subseteq E(G)$, we use $d_A(u)$ to denote the number of edges in $A$ incident with $u$.
For a vertex subset $X\subseteq V(G)$, let $N_X(u)=X\cap N_G(u)$ and $d_X(u)=|N_X(u)|$.
We use $N_G(X)$ and $N_G[X]$ to denote $\left(\bigcup_{u\in X}N_G(u)\right)\setminus X$  and $\bigcup_{u\in X}N_G[u]$, respectively.
We also write $e_G(X)$ to express the number of edges contained in the induced subgraph $G[X]$.
We say $X$ is {\it stable} if there is no edges of $G$ contained in $X$.
We often drop the subscript when the graph $G$ is clear from the context.
For a subset $S$ of vertices or edges, let $G-S$ or $G\setminus S$ be the graph obtained from $G$ by deleting every element in $S$.
Denote by $\overline{G}$ the complement graph of $G$.

Let $G$ and $H$ be graphs and $k$ be a positive integer.
Denote by $G\cup H$ the vertex-disjoint union of $G$ and $H$ and by $k\cdot G$ the vertex-disjoint union of $k$ copies of a graph $G$.
Let $G+ H$ be obtained from $G\cup H$ by adding all possible edges between $V(G)$ and $V(H)$.
For graphs $H_1,\ldots,H_k$, it is connivent to use $H_1+\ldots+H_k$ to express the graph $(H_1+\ldots+H_{k-1})+H_k$.
For a set $X$ of vertices, by $K[X]$ we mean the complete graph with the vertex set $X$.
Let $K(V_1,\ldots,V_r)$ denote the complete $r$-partite graph with parts $V_1,\ldots,V_r$.
For a graph $F$, we denote the number of copies of $F$ in a graph $G$ as $\mathcal{N}_F(G)$ (sometimes also written as $\#F(G)$).

We denote the independent set on $k$ vertices by $I_k$,  the star on $k$ vertices by $S_{k}$, the path on $k$ vertices by $P_k$, and the matching of $k$ edges by $M_k$.
For two functions $f, g:\mathbb{N}^+\to \mathbb{R}^+$,
by $f=\Omega(g)$ we mean $f\geq c \cdot g$ for a sufficiently large constant $c$,
by $f=O(g)$ we mean $f\leq d\cdot g$ for a fixed constant $d>0$,
and by $f=\Theta(g)$ we mean that $c_1\cdot g\leq f\leq c_2 \cdot g$ for fixed constants $c_2>c_1>0$.
Throughout this paper, we write $[k]$ for the set $\{1,2,\ldots,k\}$.

\subsection{Extremal results}
We introduce some classic theorems and useful lemmas needed in the following proofs.
As we discussed in the introduction, 
the Graph Removal Lemma (see e.g., Theorem 2.9 in \cite{Komlos1996}) and the Erd\H{o}s-Simonovits Stability Theorem are key to the proofs (of Theorems~\ref{Thm:Knes} and \ref{Thm:hVSt}).

\begin{theorem}[Graph Removal Lemma \cite{Komlos1996}]\label{removal lemma} Let $F$ be a graph with $f$ vertices. Then for every $\delta>0$ there is $\epsilon>0$ such that every graph with $n\geq 1/\epsilon$ vertices and at most $\epsilon n^f$ copies of $F$ can be made $F$-free by removing at most $\delta n^2$ edges.
\end{theorem}

\begin{theorem}[Erd\H{o}s-Simonovits Stability Theorem \cite{erdHos1967,erdHos1968,Simonovits1968})]\label{stability theorem}
Let $r\geq2$ and $F$ be a graph with chromatic number $r+1$.
Then for every $\delta>0$ there is $\epsilon >0$ such that every $F$-free graph $H$ with $n\geq1/\epsilon$ vertices and at least $t_r(n)-\epsilon n^2$ edges contains an $r$-partite subgraph with at least $t_r(n)-\delta n^2$ edges and
moreover, $H$ can be obtained from an extremal graph for $F$ by changing at most $\delta n^2$ edges.
\end{theorem}

Let $Z(m,n,a,b)$ be the maximum number of edges of $G\subseteq K(m,n)$ such that $G$ does not contain a copy of $K_{a,b}$ with $a$ vertices from the first class and $b$ vertices from the second class of $K(m,n)$. In 1954, K\"{o}v\'{a}ri, S\'{o}s and Tur\'{a}n \cite{Kovari1954} proved the following classic result.

\begin{theorem}[K\"{o}v\'{a}ri, S\'{o}s and Tur\'{a}n, \cite{Kovari1954}]\label{Kab in bipartite graph}
For any integers $m\geq a$ and $n\geq b$, it holds that
$$Z(m,n,a,b)\leq (b-1)^{1/a}\cdot m n ^{1-1/a} + (a-1)n$$
\end{theorem}

We need the following special form of Theorem~\ref{Kab in bipartite graph}.

\begin{lemma}\label{bipartite graph with large degree contain bipartite complete graph}
For every real $\delta>0$ and integer $m\geq 1$, there exists a real $\epsilon>0$ such that the following holds.
If $G$ is an $(m,n)$-bipartite graph where each vertex in the partite set of size $m$ has degree at least $\delta n$,
then $G$ contains a copy of $K_{\delta m,\epsilon n}$.
\end{lemma}
\begin{proof}
Take $\epsilon>0$ small enough so that $\epsilon^{1/(\delta m)}<1/m$.
Then we have $e(G)\geq \delta mn> \epsilon^{1/(\delta m)} m  n +(\delta m-1)n> (\epsilon n-1)^{1/(\delta m)}   m n ^{1-1/(\delta m)} + (\delta m-1)n$.
Now the conclusion follows directly from Theorem~\ref{Kab in bipartite graph}.
\end{proof}

The next lemma provides a handy tool for counting matchings of given size. 

\begin{lemma}\label{number of M_k with maximum degree}
Let $\epsilon\in (0,1)$ be a small constant.
Let $G$ be an $n$-vertex graph with $e(G)\geq 2k\epsilon n $ and maximum degree $\Delta(G)\leq\epsilon n$.
Then $\mathcal{N}_{M_k}(G)\geq (k-1)! (2\epsilon n)^{k}.$
\end{lemma}
\begin{proof}
For each edge $e$ in $G$, the number of copies of $M_k$ containing $e$ is at least $(2k\epsilon n-2\epsilon n)(2k\epsilon n-4\epsilon n)\ldots 2\epsilon n= (k-1)!  (2\epsilon n)^{k-1}$.
Thus we have $\mathcal{N}_{M_k}(G)\geq  2k\epsilon n (k-1)!(2\epsilon n)^{k-1} /k\geq (k-1)!(2\epsilon n)^{k}.$ The proof of Lemma~\ref{number of M_k with maximum degree} is complete.
\end{proof}

We also need the following useful lemma proved by Mubayi \cite{mubayi2010}.

\begin{lemma}[Mubayi, Lemma 4 in \cite{mubayi2010}]\label{Mubayi}
Suppose that $r\geq 2$ is fixed, $n$ is sufficiently large, $s<n$ and $n_1+\ldots +n_r=n$. If $\sum_{1\leq i<j\leq r} n_i n_j\geq t_r(n)-s$,
then $\lfloor n/r \rfloor -s\leq n_i\leq \lceil n/r \rceil +s$ for all $i\in [r]$.
\end{lemma}

\subsection{Color-$k$-critical graphs}\label{subsec:matching-critical}

\noindent In this subsection, we introduce a significant family of graphs that plays a crucial role in our proofs: the color-$k$-critical graphs.
We will also present an extremal result due to Simonovits for graphs in this family.

\begin{definition}\label{Def:k-critical}
\emph{For any positive integer $k$, a graph $G$ is called {\bf color-$k$-critical} if
\begin{itemize}
\item[(i).] there exist $k$ suitable edges whose removal decreases its chromatic number, and
\item[(ii).] deleting any $k-1$ vertices does not decrease its chromatic number.
\end{itemize}}
\end{definition}
\noindent
It is clear from the definition that any $k$ edges whose removal decreases $\chi(G)$ must form a matching of size $k$.
In particular, color-$1$-critical graphs are just color-critical graphs.\footnote{This is why we refer to this family as color-$k$-critical, as it naturally extends the concept of color-critical graphs.}

In \cite{Simonovits1974}, Simonovits determined the unique extremal graph for every color-$k$-critical graph.

\begin{definition}
Denote by $H(n,r,k)=K_{k-1}+ T_r(n-k+1)$ the $n$-vertex graph obtained by joining each vertex of the Tur\'{a}n graph $T_r(n-k+1)$ to each vertex of a copy of $K_{k-1}$. Let $h(n,r,k)=e(H(n,r,k))$.
\end{definition}

\begin{theorem}[Simonovits, Theorem 2.2 in \cite{Simonovits1974}]\label{extremal graph for mk-critical graph}
Let $k\geq 1$ and let $F$ be a color-$k$-critical graph with $\chi(F)=r+1$.
If $n$ is sufficiently large, then $H(n,r,k)$ is the unique extremal graph for $F$.
\end{theorem}

It is already known that (see, i.e., \cite{Simonovits1974,Simonovits1999})
the family of color-$k$-critical graphs is rich, including disjoint unions of cliques $K_r$, the Petersen graph, and the dodecahedron graph.
%We refer to \cite{Simonovits1999} for more information on the study of this graph family.
In Subsection~\ref{Subs:Kneser}, we show that Kneser graphs $K(t,2)$ for every $t\geq 6$ are color-$k$-critical graphs for $k=3$ (and actually we show that they are color-$3$-critical with additional nice properties).

We conclude this section with the following lemma. 
It is easy to see that the only bipartite color-$k$-critical graph is the matching of size $k$.

\begin{lemma}\label{lem:color-k-critical-is-stable}
Any non-bipartite color-$k$-critical graph is stable.
\end{lemma}

\begin{proof}
We prove the following stronger assertion that for any non-bipartite graph $F$, if the extremal graph for $F$ is unique for sufficiently large $n$, then $F$ is stable.
Together with Theorem~\ref{extremal graph for mk-critical graph}, this implies the lemma. 

Let $\chi(F)=r+1$ with $r\geq 2$. Let $n$ be sufficiently large, and $H(n)$ be the unique extremal graph for $F$. 
To prove this assertion, it is sufficient to show that every $n$-vertex $F$-free graph $G$ with $(1-o(1))\ex(n,F)$ edges can be obtained from $H(n)$ by changing at most $o(n^2)$ edges.
By Erd\H{o}s-Stone-Simonovits Theorem, $\ex(n,F)=t_r(n)+o(n^2)$.
So every $n$-vertex $F$-free graph $G$ with $(1-o(1))\ex(n,F)$ edges also contains at least $t_r(n)-o(n^2)$ edges.
By Theorem~\ref{stability theorem}, $G$ can be obtained from $H(n)$ by changing at most $o(n^2)$ edges, as desired.
\end{proof}

\section{Counterexamples to Conjecture~\ref{con-mubayi}}\label{Sec:example}
In this section, we prove Theorem~\ref{main-example} by providing a counterexample to Conjecture~\ref{con-mubayi} for every integer $q\geq 4$.
As we shall see later, this proof in fact leads to infinitely many counterexamples to Conjecture~\ref{con-mubayi}.
We will construct a non-bipartite stable graph $F$ and show that there exists a small constant $b_F>0$ such that
for any sufficiently large integer $n$ and any integer $4\leq q\leq b_F n$, $$\frac{h_F(n,q)}{q\cdot c(n,F)}\leq 1-b_F.$$
Let $k\geq 2$ be any integer.
Throughout this section, we let $A=M_k$ and $B=P_4\cup M_{k-2}$ be two fixed graphs and define $F=A+B$ (See Figure 1 (a)).
It is clear that $\chi(F)=4$.

We first explain that $F$ is a stable color-$k$-critical graph.
If we delete any $k-1$ vertices from $F$,
the resulting graph contains at least one edge in $A$ and at least one edge in $B$ and hence contains a copy of $K_4$.
Moreover, it is easy to see that removing all $k$ edges in $A$ will decrease the chromatic number by one.
Hence, $F$ is indeed color-$k$-critical.
By Lemma~\ref{lem:color-k-critical-is-stable}, we see that $F$ is also stable.

Let $X\cup V_1\cup V_2\cup V_3$ be the partition of $V(H(n,3,k))$ %(or $V(H^\prime(n,3,k))=V(I_{k-1}+T_3(n-2))$)
such that $X$ induces the clique of size $k-1$ and each $V_i$ is an independent set of size $n_i$ for $i\in [3]$,
where $\lceil(n-k+1)/3\rceil= n_1\geq n_2\geq n_3=\lfloor(n-k+1)/3\rfloor$.
Let $H_i$ be the graph obtained from $H(n,3,k)$ by adding one edge into $V_i$ for $i\in [3]$.

Throughout the rest of the proof, let $\{i,j,\ell\}=\{1,2,3\}$.
We now consider all possible embeddings of $F$ in each $H_i$.
Suppose that $H_i$ contains a copy of $F=A+B$.
We claim that
\begin{equation}\label{equ:count F}
\mbox{ either } V(A)\subseteq X\cup V_i \mbox{ and } V(B) \subseteq V_j\cup V_\ell, \mbox{ or } V(B)\subseteq X\cup V_i \mbox{ and } V(A) \subseteq V_j\cup V_\ell.
\end{equation}
To see this, first suppose that $V_j\cup V_\ell$ contains some vertices $x\in A$ and $y\in B$.
Then $H_i[V_j\cup V_\ell]$ cannot contain an edge from $A$ or from $B$;
otherwise this edge (say in $A$) together with the vertex $y$ in $B$ will form a triangle (by the definition of $F$) in $H_i[V_j\cup V_\ell]$, but $H_i[V_j\cup V_\ell]$ is bipartite, a contradiction.
Hence $H_i[V_j\cup V_\ell]$ contains at most $k$ vertices from $A$ and at most $k$ vertices from $B$.
That says, $H_i[X\cup V_i]$ must contains at least $k$ vertices from $A$ and at least $k$ vertices from $B$, and hence contains a copy of $K_{k,k}$, a contradiction.
So $V_j\cup V_\ell$ has either (i) no vertices from $A$ or (ii) no vertices from $B$.
Suppose (i) occurs.
If $V_j\cup V_\ell$ contains at most $2k-1$ vertices of $B$, then $H_i[X\cup V_i]$ contains all $2k$ vertices of $A$ and at least one vertex of $B$.
In particular, $H_i[X\cup V_i]$ contains a copy of $K_1+M_k$, but this is a contradiction.
Hence when (i) occurs, $V_j\cup V_\ell$ must contain all vertices from $B$, implying \eqref{equ:count F}.
The other case (ii) can be derived similarly. This proves \eqref{equ:count F}.

Let $c_i(n,F)$ denote the number of copies of $F$ in $H_i$.
Using \eqref{equ:count F} we can compute $c_i(n,F)$ precisely.
We note that the numbers of copies of $A$ and $B$ in $K_{k,k}$ are $k!$ and $k!k(k-1)$, respectively.
Moreover, the numbers of copies of $A$ and $B$ in $K_{k-1}+(K_2\cup I_{k-1})$ are $(k-1)!$ and $(\frac{3k}{2}-1)(k-1)(k-1)!$, respectively.\footnote{
The later one holds because that the numbers of copies of $B=P_4\cup M_{k-2}$ in $K_{k-1}+(K_2\cup I_{k-1})$ with the middle edge of $P_4$ lying inside $K_{k-1}$, between $K_{k-1}$ and $I_{k-1}$,
and between $K_{k-1}$ and $K_2$ are ${k-1 \choose 2}(k-1)!$,  $(k-1)(k-2)(k-1)!$, and $2(k-1)(k-1)!$, respectively, which add up to $(\frac{3k}{2}-1)(k-1)(k-1)!$.}
Following \eqref{equ:count F}, there are only two ways of embedding $F$ in $H_i$, which leads to
\begin{align*}
c_i(n,F)=& \left((k-1)!{n_i-2 \choose k-1}\right)\cdot \left(k!k(k-1){n_j \choose k}{n_\ell \choose k}\right)\\
&~~~ + \left((3k/2-1)(k-1)(k-1)!{n_i-2 \choose k-1}\right)\cdot \left(k!{n_j \choose k}{n_\ell \choose k}\right)\\
=&\frac{(k-1)(5k-2)}{2}\cdot (k-1)!k!\cdot {n_i-2 \choose k-1}{n_j \choose k}{n_\ell \choose k}.
\end{align*}
Since ${(x+1)-2 \choose k-1}{x \choose k}<{x-2 \choose k-1}{x+1 \choose k}$ for sufficiently large integers $x$, we have
\begin{equation}\label{eq1-for-counter-example}
c(n,F)=\min_{1\leq i\leq 3} c_i(n,F)=c_1(n,F)=\frac{(k-1)(5k-2)}{2}\cdot (k-1)!k!\cdot {n_1-2 \choose k-1}{n_2 \choose k}{n_3 \choose k}.
\end{equation}

In what follows, we will construct an $n$-vertex graph $H^*$ with $\ex(n,F)+q=e(H(n,3,k))+q$ edges which contains at most $(1-b_F)\cdot q\cdot c(n,F)$ copies of $F$.
As indicated in the beginning of this section, here we take $n$ to be sufficiently large and $q$ to be any integer at least $4$ and at most $b_F\cdot n$ for some small constant $b_F>0$.
To construct $H^*$, we first take $H^\prime=I_{k-1}+T_3(n-k+1)$ and let $V_1, V_2, V_3$ be the three partite sets of $T_3(n-k+1)$ with $n_i=|V_i|$ and $n_1\geq n_2\geq n_3$.
Let $t=q+{k\choose 2}+1$.
Now define $H^*$ to be the graph obtained from $H^\prime$ by first adding a copy of the star $S_{t}$ into $V_1$ and then removing the $k-1$ edges between the center of $S_{t}$ and $I_{k-1}$ of $H^\prime$.

First observe that indeed $H^*$ has $e(H^\prime)+(t-1)-(k-1)=e(H(n,3,k))-\binom{k-1}{2}+q+\binom{k}{2}-k=e(H(n,3,k))+q$ edges.
Note that any copy of $F$ in $H^*$ using $w\geq 2$ edges of $S_t$ must contain $w+1$ vertices of $S_t$, all vertices in $I_{k-1}$, and $3k-w$ other vertices.
So the number of copies of $F$ in $H^*$ using at least two edges from $S_t$ is $O_F(\sum_{w\geq 2} q^w n^{3k-w})=O_F(q^2n^{3k-2})$, where the inequality holds as $q/n\leq b_F$.
Next we consider the number $\mathcal{N}_1$ of copies of $F$ in $H^*$ using exactly one edge from $S_t$.
We point out that every such $F$ use all $k-1$ vertices of $I_{k-1}$ and thus the claim \eqref{equ:count F} applies when counting $\mathcal{N}_1$.
Since the $k$ edges between the center of $S_t$ and $I_{k-1}$ are deleted in $H^*$,
the number of copies of $B$ in $H^*$ using a fixed edge from $S_t$ and $k-1$ fixed vertices of $V_1$ equals $(k-1)(k-2)(k-1)!+(k-1)(k-1)!=(k-1)^2(k-1)!$.
Following \eqref{equ:count F} we have
$$\frac{\mathcal{N}_1}{t-1}=\bigg(k!\cdot (k-1)^2(k-1)!+(k-1)!\cdot k!k(k-1)\bigg)\cdot {n_1-2 \choose k-1}{n_2 \choose k}{n_3 \choose k},$$
where $t-1=q+\binom{k}{2}$.
Putting everything together, $\mathcal{N}_F(H^*)=\mathcal{N}_1+O_F(q^2n^{3k-2})$ which gives that
\begin{align}\label{eq2-for-counter-example}
\mathcal{N}_F(H^*)=\left(q+{k \choose 2}\right)\cdot (k-1)(2k-1)\cdot (k-1)!k!\cdot {n_1-2 \choose k-1}{n_2 \choose k}{n_3 \choose k}+O_F(q^2n^{3k-2}).
\end{align}
Comparing with \eqref{eq1-for-counter-example} and \eqref{eq2-for-counter-example}, we see that there exists a small constant $b_F>0$ such that
$\mathcal{N}_F(H^*)\leq (1-b_F)\cdot q\cdot c(n,F)$ as long as $\left(q+\binom{k}{2}\right)\cdot(2k-1)<(1-b_F)\cdot q\cdot \frac{5k-2}{2}$ and $q\leq b_F n.$
Solving the inequality, this shows that there exists a small constant $b_F>0$ such that
$$\mathcal{N}_F(H^*)<(1-b_F)\cdot q\cdot c(n,F) \mbox{ whenever } q \mbox{ satisfies that } (k-1)(2k-1)< q\leq  b_F n.$$
In particular, if we take $k=2$ and $F=M_2+P_4$, then $h_F(n,q)\leq \mathcal{N}_F(H^*)< (1-b_F)\cdot q\cdot c(n,F)$ holds for any $4\leq q\leq b_F n$ when $n$ is sufficiently large.
The proof of Theorem~\ref{main-example} is complete.
\QED

\bigskip

\begin{center}
\begin{tikzpicture}[scale = 0.25]
\draw (-46.5,-9)--(-46.5,5);
\draw (-51.5,-9)--(-51.5,5);
\draw (-51.5,-9) arc (180:270:0.5);
\draw (-51.5,5) arc (180:90:0.5);
\draw (-46.5,-9) arc (0:-90:0.5);
\draw (-46.5,5) arc (0:90:0.5);
\draw (-51,5.5)--(-47,5.5);
\draw (-51,-9.5)--(-47,-9.5);

\draw (-43.5,-9)--(-43.5,5);
\draw (-38.5,-9)--(-38.5,5);
\draw (-43.5,-9) arc (180:270:0.5);
\draw (-43.5,5) arc (180:90:0.5);
\draw (-38.5,-9) arc (0:-90:0.5);
\draw (-38.5,5) arc (0:90:0.5);
\draw (-43,5.5)--(-39,5.5);
\draw (-43,-9.5)--(-39,-9.5);
\filldraw[fill=gray!20]  (-46.5,-9) rectangle (-43.5,5);

\draw [line width=0.5pt](-50,3) -- (-48,3);
\draw [line width=0.5pt](-50,2) -- (-48,2);
\filldraw[fill=black] (-50,3) circle(2pt);
\filldraw[fill=black] (-50,2) circle(2pt);
\filldraw[fill=black] (-48,3) circle(2pt);
\filldraw[fill=black] (-48,2) circle(2pt);
\filldraw[fill=black] (-49,0) circle(2pt);
\filldraw[fill=black] (-49,-1) circle(2pt);
\filldraw[fill=black] (-49,-2) circle(2pt);
\draw [line width=0.5pt](-50,-5) -- (-48,-5);
\draw [line width=0.5pt](-50,-6) -- (-48,-6);
\filldraw[fill=black] (-50,-5) circle(2pt);
\filldraw[fill=black] (-50,-6) circle(2pt);
\filldraw[fill=black] (-48,-5) circle(2pt);
\filldraw[fill=black] (-48,-6) circle(2pt);

\draw [line width=0.5pt](-42,3) -- (-40,3);
\draw [line width=0.5pt](-42,2) -- (-40,2);
\filldraw[fill=black] (-42,3) circle(2pt);
\filldraw[fill=black] (-42,2) circle(2pt);
\filldraw[fill=black] (-40,3) circle(2pt);
\filldraw[fill=black] (-40,2) circle(2pt);
\filldraw[fill=black] (-41,0) circle(2pt);
\filldraw[fill=black] (-41,-1) circle(2pt);
\filldraw[fill=black] (-41,-2) circle(2pt);
\draw [line width=0.5pt](-42,-5) -- (-40,-5);
\draw [line width=0.5pt](-42,-6) -- (-40,-6);
\filldraw[fill=black] (-42,-5) circle(2pt);
\filldraw[fill=black] (-42,-6) circle(2pt);
\filldraw[fill=black] (-40,-5) circle(2pt);
\filldraw[fill=black] (-40,-6) circle(2pt);
\draw [line width=0.5pt](-42,3) -- (-42,2);
\draw node at (-45,-12){(a). For Theorem~\ref{main-example}};

\draw (-26.5,-9)--(-26.5,5);
\draw (-31.5,-9)--(-31.5,5);
\draw (-31.5,-9) arc (180:270:0.5);
\draw (-31.5,5) arc (180:90:0.5);
\draw (-26.5,-9) arc (0:-90:0.5);
\draw (-26.5,5) arc (0:90:0.5);
\draw (-31,5.5)--(-27,5.5);
\draw (-31,-9.5)--(-27,-9.5);

\draw (-23.5,-9)--(-23.5,5);
\draw (-18.5,-9)--(-18.5,5);
\draw (-23.5,-9) arc (180:270:0.5);
\draw (-23.5,5) arc (180:90:0.5);
\draw (-18.5,-9) arc (0:-90:0.5);
\draw (-18.5,5) arc (0:90:0.5);
\draw (-23,5.5)--(-19,5.5);
\draw (-23,-9.5)--(-19,-9.5);
\filldraw[fill=gray!20]  (-26.5,-9) rectangle (-23.5,5);

\draw [line width=0.5pt](-30,3) -- (-28,3);
\draw [line width=0.5pt](-30,2) -- (-28,2);
\filldraw[fill=black] (-30,3) circle(2pt);
\filldraw[fill=black] (-30,2) circle(2pt);
\filldraw[fill=black] (-28,3) circle(2pt);
\filldraw[fill=black] (-28,2) circle(2pt);
\filldraw[fill=black] (-29,0) circle(2pt);
\filldraw[fill=black] (-29,-1) circle(2pt);
\filldraw[fill=black] (-29,-2) circle(2pt);
\draw [line width=0.5pt](-30,-5) -- (-28,-5);
\draw [line width=0.5pt](-30,-6) -- (-28,-6);
\filldraw[fill=black] (-30,-5) circle(2pt);
\filldraw[fill=black] (-30,-6) circle(2pt);
\filldraw[fill=black] (-28,-5) circle(2pt);
\filldraw[fill=black] (-28,-6) circle(2pt);

\draw [line width=0.5pt](-22,3) -- (-20,3.3);
\draw [line width=0.5pt](-22,3) -- (-20,2.7);
\draw [line width=0.5pt](-22,3) -- (-20,2);
\draw [line width=0.5pt](-22,3) -- (-20,4);
\filldraw[fill=black] (-22,3) circle(2pt);
\filldraw[fill=black] (-20,4) circle(2pt);
\filldraw[fill=black] (-20,3.3) circle(2pt);
\filldraw[fill=black] (-20,2.7) circle(2pt);
\filldraw[fill=black] (-20,2) circle(2pt);
\draw [line width=0.5pt](-22,0) -- (-20,1);
\draw [line width=0.5pt](-22,0) -- (-20,0.3);
\draw [line width=0.5pt](-22,0) -- (-20,-0.3);
\draw [line width=0.5pt](-22,0) -- (-20,-1);
\filldraw[fill=black] (-22,0) circle(2pt);
\filldraw[fill=black] (-20,1) circle(2pt);
\filldraw[fill=black] (-20,0.3) circle(2pt);
\filldraw[fill=black] (-20,-0.3) circle(2pt);
\filldraw[fill=black] (-20,-1) circle(2pt);
\filldraw[fill=black] (-21,-3.5) circle(2pt);
\filldraw[fill=black] (-21,-1.5) circle(2pt);
\filldraw[fill=black] (-21,-2.5) circle(2pt);
\draw [line width=0.5pt](-22,-5) -- (-20,-5.5);
\draw [line width=0.5pt](-22,-5) -- (-20,-6.5);
\draw [line width=0.5pt](-22,-5) -- (-20,-4.5);
\draw [line width=0.5pt](-22,-5) -- (-20,-3.5);
\filldraw[fill=black] (-22,-5) circle(2pt);
\filldraw[fill=black] (-20,-4.5) circle(2pt);
\filldraw[fill=black] (-20,-5.5) circle(2pt);
\filldraw[fill=black] (-20,-6.5) circle(2pt);
\filldraw[fill=black] (-20,-3.5) circle(2pt);
\draw [line width=0.5pt](-22,3) -- (-22,0);

\draw node at (-25,-12){(b). For Theorem~\ref{Thm:hVSt}};
\end{tikzpicture}

\medskip

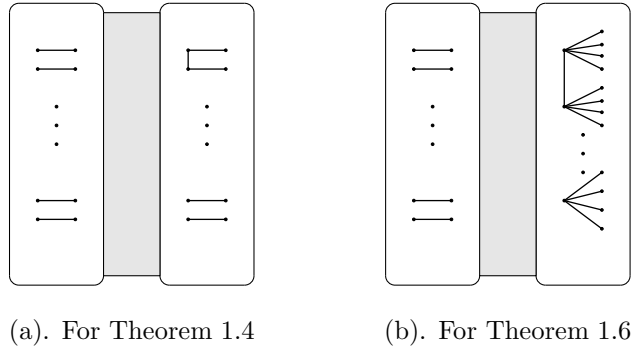
Figure 1. Examples for color-$k$-critical graphs

\end{center}

\section{Properties on supersaturated graphs}\label{sec_properties-of-H}
In the rest of this paper, let $F$ be a color-$k$-critical graph on $f$ vertices with $\chi(F)=r+1$ where $r\geq 2$.
This section aims to establish some quantitative and structural properties for graphs with the minimum number of copies of $F$ subject to given numbers of vertices and edges.

\subsection{Basic properties}
We first present some lemmas on the minimum number of copies of $F$ obtained from some well-characterized graphs by adding few edges,
which generalize similar lemmas proved in \cite{mubayi2010, pikhurko2017}.

Recall Theorem~\ref{extremal graph for mk-critical graph} that for sufficiently large $n$,
$H(n,r,k)$ is the unique $n$-vertex extremal graph for $F$.
The coming two lemmas concerns quantitative properties of $c(n,F)$,
which, in this case, denotes the minimum number of copies of $F$ obtained from $H(n,r,k)$ by adding one new edge.%\footnote{The number of copies of $F$ maybe different when we put the added edge to different parts of $H(n,r,k)$.}

Let $n_1,\ldots, n_r$ be positive integers satisfying $\sum_{i=1}^r n_i=n-k+1$ and
let $H(n_1,\ldots, n_r)$ be the graph obtained from $K_{k-1}+K(V_1,\ldots,V_r)$ by adding a new edge $xy$ into $V_1$ where each $|V_i|=n_i$.
Let $c(n_1,\ldots,n_r;F)$ be the number of copies of $F$ contained in $H(n_1,\ldots, n_t)$.

\begin{lemma}\label{counting F in H(n,p,k) with adding an edge}
There are positive constants $\alpha_F,\beta_F$ such that if $n$ is sufficiently large, then
$$|c(n,F)-\alpha_F n^{f-k-1}|<\beta_F n ^{f-k-2}.$$
In particular, $\frac{1}{2}\alpha_F n^{f-k-1} <c(n,F)<2\alpha_F n^{f-k-1}.$
\end{lemma}
\begin{proof}
Let $n_1\leq \ldots\leq n_r\leq n_1+1$ be integers satisfying $\sum_{i=1}^r n_i=n-k+1$. Then we have
$$
c(n,F)=\min\{c(n_1,\ldots,n_r;F),c(n_r,\ldots,n_1;F)\}.
$$
Since $F$ is color-$k$-critical, there exist $k-1$ vertices $x_1,\ldots,x_{k-1}$ and an edge $uv$ such that $F-\{x_1,\ldots,x_{k-1}, uv\}$ has a proper $r$-coloring $c$.
We call $\{x_1,\ldots,x_{k-1},uv\}$ a {\it critical-$k$-tuple} of $F$.
Recall the definition of $H(n_1,\ldots, n_r)$ and the edge $xy$ in $H(n_1,\ldots, n_r)$.
Then an edge preserving injection of $F$ into $H(n_1,\ldots, n_r)$ is obtained by choosing a critical-$k$-tuple $\{x_1,\ldots,x_k,uv\}$ of $F$,
mapping $x_1,\ldots,x_k$ to the vertices of the $K_{k-1}$ of $H(n_1,\ldots, n_r)$, mapping $uv$ of $F$ to $xy$ of $H(n_1,\ldots, n_r)$, and then mapping the remaining vertices of $F$ properly.
This mapping corresponds to a proper coloring $c$ of $F-\{x_1,\ldots,x_{k-1}, uv\}$.
Let $x_c^i$ be the number of vertices of $F-\{x_1,\ldots,x_{k-1}, uv\}$ after excluding $u,v$ that receive color $i$ under $c$.
Let $Aut(F)$ denote the number of automorphisms of $F$.
Let $\mathcal{X}$ be the set of all critical-$k$-tuples of $F$ and $\mathcal{Y}(X)$ be the set of all proper colorings of $F-X$ for any $X\in \mathcal{X}$.
Hence, we obtain
\begin{equation}\label{number of F in h(n,p,k)}
c(n_1,\ldots,n_r;F)=\frac{1}{Aut(F)}\sum_{X\in \mathcal{X}}\sum_{c\in \mathcal{Y}(X)}(k-1)!2(n_1-2)_{x_{c}^1}\prod^r_{i=2}(n_i)_{x_{c}^i},
\end{equation}
where $(n)_k =n!/(n-k)!$. Since each $n_i$ satisfies $|n_i-\frac{n-k+1}{r}|\leq 1$,
we see that $c(n_1,\ldots,n_r;F)$ is a polynomial in $n$ of degree $f-k-1$, so is $c(n,F)$.
The lemma follows.
\end{proof}

\begin{lemma}\label{estimate c(n,f)}
There exist constants $\theta_F$ and $\eta_F$ such that the following holds for sufficiently large $n$.
Let $\sum_{i=1}^rn_i=\sum_{i=1}^{r}n^\prime_i=n-k+1$ and $c(n,F)=c(n^\prime_1,\ldots,n^\prime_r;F)$.
Let $a_i=n_i-n^\prime_i$ for each $i\in [r]$ and $A=\max\{|a_i|: i\in [r]\}$. Then
$|c(n_1,\ldots,n_r;F)-c(n,F)-\theta_F a_1 n^{f-k-2}|\leq \eta_F A^2 n^{f-k-3}.$
\end{lemma}
\begin{proof}
Note that we have $|n_i^\prime-\frac{n-k+1}{r}|\leq 1$ for each $i\in [r]$.
The assertion holds trivially for $A=0$, hence we can assume that $A\geq 1$.
By the Taylor expansion about ($n_1^\prime,\ldots,n^\prime_r$),
\begin{equation}\label{taylor expansion}
c(n_1^\prime+a_1,\ldots,n_r^\prime+a_r;F)-c(n_1^\prime,\ldots,n_r^\prime;F)-\sum^{r}_{j=1}a_j\frac{\partial c}{\partial_j}(n_1^\prime,\ldots,n^\prime_r)
\end{equation}
is a polynomial of degree at most $f-k-3$ with variables $n^\prime_i$ in which every monomial contains at least two $a_i$'s,
thus this is $O(A^2 n^{f-k-3})$. Furthermore, since $\frac{\partial c}{\partial_i}(n_1^\prime,\ldots,n_r^\prime)$ is a polynomial of degree $f-k-2$
and $|n_i^\prime-\frac{n-k+1}{r}|\leq 1$ for each $i\in[r]$, we have
$$
\left|\frac{\partial c}{\partial_i}(n_1^\prime,\ldots,n_r^\prime)-\frac{\partial c}{\partial_i}\left(\frac{n-k+1}{r},\ldots,\frac{n-k+1}{r}\right)\right|=O(n^{f-k-3}).
$$
Thus (\ref{taylor expansion}) remains within $O(A^2 n^{f-k-3})$ if we replace the term $\sum^{r}_{j=1}a_j\frac{\partial c}{\partial_j}(n_1^\prime,\ldots,n^\prime_r)$ in (\ref{taylor expansion}) by
$$
\sum_{j=1}^{r}a_j\frac{\partial c}{\partial_j}\left(\frac{n-k+1}{r},\ldots,\frac{n-k+1}{r}\right)$$
$$=a_1\left(\frac{\partial c}{\partial_1}\left(\frac{n-k+1}{r},\ldots,\frac{n-k+1}{r}\right)-\frac{\partial c}{\partial_2}\left(\frac{n-k+1}{r},\ldots,\frac{n-k+1}{r}\right)\right)
$$
where we used the facts that $\sum_{i=1}^{r}a_i=0$ and, by symmetry, all partial derivatives for $j\in \{2,\ldots,r\}$ are equal to each other (this fact can be seen from \eqref{number of F in h(n,p,k)}).
Now if we let $\theta_F$ be the coefficient of $n^{f-k-2}$ in $\frac{\partial c}{\partial_1}(\frac{n-k+1}{r},\ldots,\frac{n-k+1}{r})-\frac{\partial c}{\partial_2}(\frac{n-k+1}{r},\ldots,\frac{n-k+1}{r})$,
then the lemma follows.
\end{proof}

Let $d(n,F)$ be the minimum number of copies of $F$ in the graph obtained from $T_r(n)$ by adding a copy of $M_k$ to one partite set of $T_r(n)$.
By a proof similar to that of Lemma~\ref{counting F in H(n,p,k) with adding an edge},
we can show the following lemma and in particular, $d(n,F)$ is a polynomial in $n$ of degree $f-2k$.

\begin{lemma}\label{counting F in Tp(n) with adding a matching}
 There are positive constants $\alpha^\prime_F,\beta^\prime_F$ such that if $n$ is sufficiently large, then
$$|d(n,F)-\alpha^\prime_F n^{f-2k}|<\beta^\prime_F n ^{f-2k-1}.$$
In particular, $\frac{1}{2}\alpha^\prime_F n^{f-2k} <d(n,F)<2\alpha^\prime_F n^{f-2k}.$
\end{lemma}

We say an $n$-vertex $r$-partite graph $G$ with a partition $V(G)=\bigcup_{i=1}^{r}V_i$ is {\it $\delta$-equivalence}
 if $|V_1|=\ldots=|V_r|$ and each vertex is adjacent to at least $(1-\delta)n/r$ vertices in each of other partite sets.

\begin{lemma}\label{couting F in near complete graph within a mathing}
Let $0\leq\delta\ll1$. Let $G^\prime$ be the graph obtained from an $n$-vertex $\delta$-equivalence $r$-partite graph $G$ by adding a copy of $M_k$ into one partite set of $G$.
Then there is a positive constant $\gamma$ depending on $F$ and $\delta$ such that
$\mathcal{N}_F(G^\prime)\geq d(n,F)-\gamma n^{f-2k}$, where $\gamma\to 0$ as $\delta\to 0$.
\end{lemma}
\begin{proof}
Without loss of generality, let $u^\prime_1v^\prime_1,\ldots,u^\prime_kv^\prime_k$ denote the $k$-matching added in $G[V_1]$.
Since $F$ is color-$k$-critical, there exist $k$ edges $u_1v_1,\ldots,u_kv_k$ (call it a {\it critical-matching} of $F$) such that after deleting them, the resulting graph has a proper $r$-coloring $c$.
Let $t_c^i$ be the number of vertices of $F-\{u_1v_1,\ldots,u_kv_k\}$ that receive color $i$ under $c$.
Let $\mathcal{X}$ be the set of all critical-matchings of $F$ and $\mathcal{Y}(M)$ be the set of all proper $r$-colorings of $F-M$ for any $M\in \mathcal{X}$.
Since $d_{V_{j}}(x)\geq (1-\delta)|V_j|$ for each $x\in V_i$ and $j\neq i$,
\begin{align*}
\mathcal{N}_F(G^\prime)\geq &\frac{1}{Aut(F)}\sum_{M\in \mathcal{X}}\sum_{c\in \mathcal{Y}(M)}2^kk!(n/r-2k)_{t_c^1}\prod^r_{i=2}((1-f\delta)n/r)_{t_c^i},\\
\geq&\frac{1}{Aut(F)}\sum_{M\in \mathcal{X}}\sum_{c\in \mathcal{Y}(M)}2^kk!\left(\frac{n-f\delta n}{r}-2k\right)_{t_c^1}\prod^r_{i=2}\left(\frac{n-f\delta n}{r}\right)_{t_c^i}\\
\geq &d(n-f\delta n,F)\geq d(n,F)-\gamma n^{f-2k},
\end{align*}
where the last inequality holds because $n$ is sufficiently large and $d(n,F)$ is a polynomial in $n$ of degree $f-2k$.
Here, $\gamma\to 0$ as $\delta\to 0$. The proof of Lemma~\ref{couting F in near complete graph within a mathing} is complete.
\end{proof}

\subsection{Refined properties}\label{subsec:refined}
Let $F$ be a color-$k$-critical graph on $f$ vertices with $\chi(F)=r+1$ where $r\geq 2$.
Throughout this subsection, we use the following constants satisfying the given hierarchy:
$$ 1\gg\epsilon\gg\epsilon_{14}\gg  \epsilon_{13}\gg\ldots\gg\epsilon_2 \gg\epsilon_1\gg\delta \gg\frac{1}{n}.$$
Let $1\leq q\leq \delta n$ and $H$ be an $n$-vertex graph with ex$(n,F)+q$ edges and minimum number of copies of $F$.
We will show some refined properties on $H$, which are important in the coming sections.

Let $H(n,r,k,q)$ be the graph obtained from $H(n,r,k)$ by adding a copy of $S_{q+1}$ into one part of $H(n,r,k)$ such that the number of copies of $F$ using exactly one edge from $S_{q+1}$ is $c(n,F)$.
It is clear (by considering the definition of color-$k$-critical) that any copy of $F$ in $H(n,r,k,q)$ must use the center of the $S_{q+1}$ as well as the $k-1$ vertices of degree $n-1$ (call them the {\it top} vertices of $H(n,r,k)$).
If a copy of $F$ in $H(n,r,k,q)$ uses $t\geq 2$ edges of $S_{q+1}$, then except the $t+1$ vertices of $S_{q+1}$ and the $k-1$ top vertices, it uses $f-k-t$ many other vertices.
So the number of copies of $F$ in $H(n,r,k,q)$ using at least two edges of $S_{q+1}$ is at most $O_F\left(\sum_{t=2}^q \binom{q}{t}\cdot n^{f-k-t}\right)=O_F(q^2)\cdot n^{f-k-2}$, where we use $q/n\leq \delta$.
Hence, we have
\begin{equation}\label{low bound for H}
\mathcal{N}_F(H)\leq \mathcal{N}_F(H(n,r,k,q))=q\cdot c(n,F)+O_F(q^2)\cdot n^{f-k-2}\leq \epsilon_1 n^{f-k}=\left(\frac{\epsilon_1}{n^k}\right)n^f.
\end{equation}
Since $n$ is sufficiently large, by Theorem~\ref{removal lemma} there are at most $\epsilon_2 n^2$ edges of $H$ whose removal results in a graph $H^{\prime}$ with no copies of $F.$
Since $e(H^\prime)>t_r(n)-\epsilon_2n^2$, by Theorem~\ref{stability theorem},
we conclude that there is an $r$-partition of $V(H)=V(H^\prime)$ such that the total number of edges in $H^\prime$ (also in $H$) between two parts is at least $t_r(n)-\epsilon_3 n^2$.

Fix an $r$-partition $V(H)=V_1\cup\ldots\cup V_r$ which maximizes $|E(H)\cap E(K(V_1,\ldots,V_r))|$.\footnote{We will call such an $r$-partition of $V(H)$ as a {\it max-cut} of $H$.}
By the previous paragraph, we have $|E(H)\cap E(K(V_1,\ldots,V_r))|\geq t_r(n)-\epsilon_3 n^2$.
Let $|V_i|=n_i$ for $i\in [r]$ with $n_1\geq \ldots \geq n_r$.
Using Lemma~\ref{Mubayi}, we can derive that
\begin{equation}\label{v_i}
\left(\frac{1}{r}-\epsilon_4\right)n\leq|V_i|\leq\left(\frac{1}{r}+\epsilon_4\right)n.
\end{equation}
Let $B=E(H)\setminus E(K(V_1,\ldots,V_r))$ and $M=E(K(V_1,\ldots,V_r))\setminus E(H)$. Then we have
\begin{equation}\label{B}
|B|=e(H)-|E(H)\cap E(K(V_1,\ldots,V_r))|\leq (\ex(n,F)+q)-(t_r(n)-\epsilon_3 n^2)\leq \epsilon_4n^2
\end{equation}
and
\begin{equation}\label{M}
|M|=e(K(V_1,\ldots,V_r))-|E(H)\cap E(K(V_1,\ldots,V_r))|\leq t_r(n)-(t_r(n)-\epsilon_3 n^2)\leq\epsilon_4n^2.
\end{equation}

\medskip

\noindent {\bf Claim \rom{1}.} There exist exactly $k-1$ vertices $x_1,\ldots, x_{k-1}$ of degree $d_H(x_i)\geq n-\epsilon_9n$.

\begin{proof}
We need to introduce some definitions first.
Let $U_i\subseteq V_i$ be the set of vertices satisfying $d_M(v)\geq 2\epsilon_6 n$ and $U^\prime_i\subseteq U_i$ be the set of vertices satisfying $d_B(v)\geq d_M(v)/2\geq \epsilon_6 n$.
Let $U=\bigcup^{r}_{i=1}U_i$ and $U^\prime=\bigcup^{r}_{i=1}U^\prime_i$.
Using \eqref{M}, we have
\begin{equation}\label{s prime is small}
|U^\prime|\leq|U|\leq \frac{2|M|}{\epsilon_6 n}\leq \frac{\epsilon_4 n^2}{\epsilon_6 n}\leq  \epsilon_5n,
\end{equation}
where we choose $\epsilon_4\leq \epsilon_5\epsilon_6 $.
Let $T_i$ be the set of vertices of $V_i\setminus U_i$ with $d_B(v)\geq \epsilon_6n$ and $T=\bigcup_{i=1}^rT_i$.
By the max-cut $V_1\cup\ldots\cup V_r$, we have $d_{V_{j}}(x)\geq d_B(x)\geq \epsilon_6n$ for any $x\in U^\prime\cup T$ and any $j\in [r]$.

Let $W$ be a maximum subset in $U^\prime \cup T $ such that $|\bigcap_{x\in W}N_{V_{j}}(x)| \geq \epsilon_7 n$ holds for at least $r-1$ indexes $j\in [r]$.
We will show that $|W|\leq k-1$.
Suppose for a contradiction that $|W|\geq k$.
Without loss of generality, there exists a set $W=\{x_1,\ldots, x_k\}$ such that $|\bigcap_{x\in W}N_{V_{j}}(x)| \geq \epsilon_7 n$ for $j=2,\ldots,r$.
By (\ref{s prime is small}), there are at least $(\epsilon_6n-\epsilon_5n)_k\geq\epsilon_5n^k$ copies of $M_{k}$ in $H[W,V_1\setminus U_1]$.
For each such copy of $M_k$, by (\ref{s prime is small}) again we can choose $(\epsilon_7-\epsilon_5)n-2k$ new vertices in $V_1\setminus U_1$ and $(\epsilon_7-\epsilon_5)n$ vertices in $V_i\setminus U_i$ for $i=2,\ldots,r$ such that the $r$-partite subgraph of $H\cap K(V_1,\ldots,V_r)$ induced by those vertices is $\delta^*$-equivalence, where $\delta^*=2\epsilon_6/(\epsilon_7-\epsilon_5)$.
By Lemmas~\ref{counting F in Tp(n) with adding a matching} and \ref{couting F in near complete graph within a mathing},
each copy of $M_{k}$ is contained in at least $\frac{1}{2}d(m,F)$ copies of $F$, where $m=(\epsilon_7-\epsilon_5)rn$.
Moreover, those copies of $F$ contain only $k$ edges in $H[W,V_1\setminus U_1]$.
Hence, there are at least $\epsilon_{5}n^k \cdot \frac{1}{2}d(m,F)> \epsilon_1 n^{f-k}$
copies of $F$, a contradiction (we choose $\epsilon_{5}(\epsilon_{7}-\epsilon_5)^{f-2k}\gg \epsilon_{1}$).
Thus, $|W|\leq k-1$. In particular, this implies that for each $i\in [r]$,
\begin{equation}\label{eq 1 for claim 1}
\mbox{ there are at most $k-1$ vertices $x$ in $V_i\setminus U_i$ with $d_B(x)\geq \epsilon_6 n$, i.e., $|T_i|\leq k-1$.}\footnote{This is because every $x\in V_i\setminus U_i$ has $d_{V_j}(x)\geq |V_j|-d_M(x)\geq (1/r-3\epsilon_6)n$ for each $j\in [r]\backslash \{i\}$.}
\end{equation}
Furthermore, for each $i\in[r]$, we have
\begin{equation}\label{eq 2 for claim 1}
e_H(V_i\setminus (U_i\cup T_i))\leq 2k \epsilon_6 n.
\end{equation}
Otherwise, by Lemma~\ref{number of M_k with maximum degree} there are at least $(k-1)!  (2\epsilon_6 n)^{k}$ copies of $M_k$ in $H[V_i\setminus (U_i\cup T_i)]$.
Together with the $2k$ vertices of a fixed copy of these $M_k$, we can choose $(\epsilon_7-\epsilon_5)n-2k$ vertices in $V_i\setminus (U_i\cup V(M_k))$ and $(\epsilon_7-\epsilon_5)n$ vertices in $V_j\setminus U_j$ for each $j\in [r]\backslash \{i\}$ to form an $r$-partite $\delta^*$-equivalence subgraph of $H\cap K(V_1,\ldots,V_r)$.
Similarly as before, by Lemma~\ref{couting F in near complete graph within a mathing}, there are at least $(k-1)! (2\epsilon_6 n)^k\cdot \frac{1}{2}d(m,F) > \epsilon_1 n^{f-k}$  copies of $F$ in $H$ (we choose $\epsilon^k_{6}(\epsilon_{7}-\epsilon_5)^{f-2k}\gg \epsilon_{1}$), a contradiction to \eqref{low bound for H}.

Next we show $0\leq|U^\prime|<4^{r-1}rk$.
Suppose not. Then $|U^{\prime}|\geq 4^{r-1}rk$ and without loss of generality, we may assume that $|U^\prime_1|\geq 4^{r-1}k$.
For each vertex $y\in U^\prime_1$ and each $j=2,\ldots,r$, we have
$$d_{V_{j}}(y)\geq d_B(y)\geq d_M(y)/2\geq (|V_j|-d_{V_{j}}(y))/2\geq (n/r-\epsilon_4 n- d_{V_{j}}(y)) /2.$$
Hence, $d_{V_{j}}(y)\geq n/3r-\epsilon_5n>(1/4)|V_j|$.
Take a subset $X_1\subseteq U^\prime$ of size $4^{r-1}k$.
Applying Lemma~\ref{bipartite graph with large degree contain bipartite complete graph} (with $\delta=1/4, m=|X_1|$ and $G=H[X_1,V_2]$ there),
one can find a subset $X_2\subseteq X_1$ of size $4^{r-2}k$ such that there are at least $\epsilon_7n$ common neighbours of vertices of $X_2$ in $V_2$.
Recursively applying Lemma~\ref{bipartite graph with large degree contain bipartite complete graph} (with $\delta=1/4, m=|X_i|$ and $G=H[X_i,V_{i+1}]$ for $2\leq i<r$),
there are sets $X_r\subseteq \ldots \subseteq X_1\subseteq U^\prime_1$ with $|X_i|=4^{r-i}k$ such that the common neighbours of vertices of $X_i$ in $V_i$ for each $i=2,...,r$ is at least $\epsilon_7n$.
But $k=|X_r|\leq |W|\leq k-1$, a contradiction to the property on $|W|$.

Let $X=T\cup U^\prime$ and $\widehat{V}_i=V_i\setminus X$.
Then by \eqref{eq 1 for claim 1}, $|X|\leq |T|+|U^\prime|\leq r(k-1)+ 4^{r-1}rk< (4^{r-1}+1)rk$.
Using \eqref{eq 2 for claim 1} and the definition of $U_i\setminus U_i'$, we have
\begin{align*}\label{equ:sum-Vi-wide}
\sum_{i=1}^{r}e_H(\widehat{V}_i)&\leq \sum_{i=1}^{r}\left(e_H(V_i\setminus (T_i\cup U_i))+\sum_{x\in U_i\setminus U_i'}d_{V_i}(x)\right)\\
&\leq \sum_{i=1}^{r}\left(e_H(V_i\setminus (T_i\cup U_i))+\sum_{x\in U_i\setminus U_i'} d_M(x)/2\right)\\
&\leq 2k \epsilon_6 rn+|M\cap E(K(\widehat{V}_1,\ldots,\widehat{V}_r))|+|U|\cdot |X|.
\end{align*}
Let $X_r\subseteq \ldots \subseteq  X_1\subseteq X_0=X$ be a sequence of subsets of $X$ such that
$X_i$ is a maximum subset of $X_{i-1}$ whose common neighbors in $\widehat{V}_i$ is at least $\epsilon_7n$ for all $i\in [r]$.\footnote{Note that in this step we can prefix any ordering of $\widehat{V}_1, \widehat{V}_2,...,\widehat{V}_r$. That is, for any permutation $\pi: [r]\to [r]$ we can require that $X_i$ is a maximum subset of $X_{i-1}$ whose common neighbors in $\widehat{V}_{\pi(i)}$ is at least $\epsilon_7n$ for all $i\in [r]$.}
By the property on $|W|$, $|X_r|\leq |X_{r-1}| \leq k-1$.
Since $H[X_{i-1},\widehat{V}_i]$ is $K_{|X_i|+1,\epsilon_7 n }$-free, by Theorem~\ref{Kab in bipartite graph},
\begin{align*}
e(H[X_{i-1},\widehat{V}_i])\leq &z\big(|X_{i-1}|,|\widehat{V}_i|,|X_i|+1, \epsilon_7 n\big)
\leq \left(\epsilon_7 n-1\right)^{\frac{1}{|X_i|+1}}|X_{i-1}||\widehat{V}_i|^{1-\frac{1}{|X_i|+1}}+|X_i| |\widehat{V}_i|
\leq (|X_i|+\epsilon_8)|\widehat{V}_i| .
\end{align*}
This implies that $e(H[X,H-X])= \sum^r_{i=1}e(H[X,\widehat{V}_i])$ is at most
\begin{align*}
& \sum^r_{i=1} \left(|X\setminus X_{i-1}|\cdot |\widehat{V}_i|+e(H[X_{i-1},\widehat{V}_i])\right) \leq \sum^r_{i=1}(|X|-|X_{i-1}|+|X_i|+\epsilon_8)\cdot |\widehat{V}_i|\\
\leq & ((r-1)|X|+|X_r|+r\epsilon_8)\cdot(1/r+\epsilon_4)n\leq ((r-1)|X|+|X_r|)\cdot n/r+O(\epsilon_8n).
\end{align*}
Putting everything above together, if $|X_{r}|\leq k-2$, then we can reach a contradiction as follows
\begin{align*}
 e(H)= & e(H[X])+e(H[X,H-X])+e(H-X)\\
\leq & |X|^2/2+e(H[X,H-X])+ e(K(\widehat{V}_1,\ldots,\widehat{V}_r))-|M\cap E(K(\widehat{V}_1,\ldots,\widehat{V}_r))|+\sum_{i=1}^{r}e_H(\widehat{V}_i)\\
\leq & |X|^2/2+\big((r-1)|X|+|X_r|\big)\cdot n/r+O(\epsilon_8n)+t_r(n-|X|)+2k\epsilon_6 rn+|U|\cdot |X|\\
\leq & t_r(n)+(k-2)n/r+O(\epsilon_8 n)<h(n,r,k)< h(n,r,k)+q=e(H),
\end{align*}
where we use $|X|<(4^{r-1}+1)rk$, $|U|\leq \epsilon_5n$ from \eqref{s prime is small}, and the fact that $h(n,r,k)=t_r(n)+(k-1)n/r+O(1)$.
This shows that $|X_{r}|=|X_{r-1}|=k-1$ and $X_r=X_{r-1}$.

Lastly, we show that any vertex in $X_{r-1}$ has degree at least $n-\epsilon_9n$ in $H$.
Suppose for a contradiction that there is an $x\in X_{r-1}$ with $d_H(x)< n-\epsilon_9n$.
Then there exists some $V_j$ with $$d_{V_j}(x)\leq |V_j|-\epsilon_9n/r=|\widehat{V}_{j}|+|T_j\cup U_j'|-\epsilon_9n/r\leq |\widehat{V}_{j}|-\epsilon_9n/2r,$$
implying that $e(H[X_{r-1},\widehat{V}_{j}])\leq (k-1)n-\epsilon_9n/2r$.
Now fix a permutation $\pi:[r]\to [r]$ with $\pi(r)=j$ and find a new sequence $X'_r\subseteq \ldots \subseteq  X'_1\subseteq X'_0=X$ such that
$X'_i$ is a maximum subset of $X'_{i-1}$ whose common neighbors in $\widehat{V}_{\pi(i)}$ is at least $\epsilon_7n$ for all $i\in [r]$.
It is clear that the overall conditions on $X'_r$ remain the same, implying that $X'_r=X_r$ and thus $X'_{r-1}=X_{r-1}$.
We then can repeat the exactly same estimations as above (just using $e(H[X'_{i-1},\widehat{V}_{\pi(i)}])$ instead of $e(H[X_{i-1},\widehat{V}_i])$),
except now we have $e(H[X'_{r-1},\widehat{V}_{\pi(r)}])=e(H[X_{r-1},\widehat{V}_{j}])\leq (k-1)n-\epsilon_9n/2r$
which is better than the previous upper bound $e(H[X_{r-1},\widehat{V}_r])\leq (k-1+\epsilon_8)n$ we used.
Repeating the above estimations, we have
\begin{align*}
e(H)\leq & |X|^2/2+e(H[X,H-X])+ e(K(\widehat{V}_1,\ldots,\widehat{V}_r))-|M\cap E(K(\widehat{V}_1,\ldots,\widehat{V}_r))|+\sum_{i=1}^{r}e_H(\widehat{V}_i)\\
\leq & |X|^2/2+\big((r-1)|X|+|X_r|\big)\cdot n/r-\epsilon_9n/2r+O(\epsilon_8n)+t_r(n-|X|)+2k\epsilon_6 rn+|U|\cdot |X|\\
\leq & t_r(n)+(k-1)n/r-\epsilon_9n/4r <h(n,r,k)< h(n,r,k)+q=e(H),
\end{align*}
a contradiction. This completes the proof of Claim~\rom{1}. \end{proof}

Throughout the rest of this section, we denote $X=\{x_1,\ldots,x_{k-1}\}$ from Claim~\rom{1}.
Let $H^*=K[X]+K(V_1\setminus X,\ldots,V_r\setminus X).$
Let $B^*=E(H)\setminus E(H^*)$ and call edges in $B^*$ {\bf bad}.
Let $M^*=E(H^*)\setminus E(H)$ and call edges in $M^*$ {\bf missing}.
Then $|B^*|-|M^*|=e(H)-e(H^*)=h(n,r,k)+q-e(H^*)\geq q$.
We point out that by definition, $B^*\subseteq B$ and thus $|B^*|\leq |B|\leq \epsilon_4 n^2$ by \eqref{B}.
The next claim gives a significant improvement on the upper bound of $|B^*|$.

\medskip

\noindent {\bf Claim~\rom{2}.} It holds that $|B^\ast|\leq \epsilon_5 n$.

\begin{proof}
Recall the definition of $d(n,F)$ (see the paragraph before Lemma~\ref{counting F in Tp(n) with adding a matching}).
We also need to introduce the following parameter, which will be repeatedly used in the proof later.
\begin{definition}\label{Def:F(e)}
For each $e\in B^*$, let $F(e)$ denote the number of copies of $F$ in $H$ containing $e$ as the unique edge from $B^*$.
\end{definition}
We now partition $B^\ast=B_1\cup B_2$, where $B_1=\left\{e\in B^\ast: F(e)>(1-\epsilon)d(n,F)(n/r)^{k-1}\right\}.$

First we demonstrate that to prove this claim, it suffices to show that $|B_1|>(1-\epsilon)|B^\ast|.$
Suppose that $|B_1|>(1-\epsilon)|B^\ast|$ and $|B^\ast|>\epsilon_5 n$ (for a contradiction).
Then the number of copies of $F$ in $H$ is
\begin{align*}
\mathcal{N}_F(H)&\geq  \sum_{e\in B_1}F(e)\geq |B_1|(1-\epsilon)d(n,F)(n/r)^{k-1}\\
&\geq(1-\epsilon)^2(\epsilon_5 n)\cdot d(n,F)\cdot (n/r)^{k-1} >q\cdot c(n,F)+O(q^2)n^{f-k-2}\geq \mathcal{N}_F(H),
\end{align*}
where the second last inequality holds strictly as $\epsilon_5 n\gg \delta n\geq q$ and the last inequality follows from \eqref{low bound for H}.
This is a contradiction.
So our goal is to show $|B_1|>(1-\epsilon)|B^\ast|,$ or equivalently $|B_2|<\epsilon |B^\ast|$.

Suppose to the contrary that $|B_2|\geq \epsilon |B^\ast|$.
Fix an arbitrary $e\in B_2$, without loss of generality, say $e\in V_1$.
By a {\it potential} copy $F$ (with respect to $e$), we mean a copy of $F$ whose edges are from $\{e\}\cup E(H^*)$.
Clearly, any potential copy $F$ must contain $X$ and we call it {\it strong} if all its edges incident to $X$ are in $H$.
Let $\mathcal{F}_e$ denote the set of strong potential copies of $F$.
Every vertex $x_i\in X$ has degree at least $n-\epsilon_9 n$.
So in $V_1\cup X$ there are at least $((1-\epsilon_{10})(n/r))^{k-1}$ copies of $M_k$ containing $e$ and $X$,
implying that $|\mathcal{F}_e|\geq ((1-\epsilon_{10})(n/r))^{k-1}d(m,F)$, where $m\geq (1-\epsilon_{10}) n$.

Let $M^\prime=\{e'\in M^*: V(e')\cap X=\emptyset\}$.
If a copy $F$ in $\mathcal{F}_e$ is not a copy in $H$, then it must contain a missing edge $e'$ in $M^*=E(H^*)\setminus E(H)$;
furthermore, this missing edge $e'$ must be in $M^\prime$.
Hence, there are at least $(\epsilon/2)d(n,F)(n/r)^{k-1}$ copies of $F\in \mathcal{F}_e$ containing a missing edge in $M^\prime$; otherwise
$$F(e)>((1-\epsilon_{10})(n/r))^{k-1}d(m,F)-(\epsilon/2)d(n,F)\left(n/r\right)^{k-1}>(1-\epsilon)d(n,F)\left(n/r\right)^{k-1},$$
a contradiction to the definition of $e\in B_2$.

Let $e'\in M^\prime$ with $V(e')\cap V(e)=\emptyset$.
Any copy $F \in \mathcal{F}_e$ containing $e'$ has exactly $f-k-3$ vertices not in $V(e)\cap V(e')\cup X$.
So there are at most $O_F(n^{f-k-3})$ copies of $F \in \mathcal{F}_e$ containing $e'$.
By Lemma~\ref{counting F in Tp(n) with adding a matching} and the fact that $|M'|\leq |M^*|\leq |B^*|\leq \epsilon_4 n^2$,
there are at most $O_F(n^{f-k-3})|M^\prime|\leq O(\epsilon_4)n^{f-k-1}\leq (\epsilon/4)d(n,F)(n/r)^{k-1}$ copies of $F\in \mathcal{F}_e$ containing a missing edge $e'\in M^\prime$ that does not intersect $e$.

Combining the conclusions of the above two paragraphs, we can derive that
for any $e\in B_2$, there are at least $(\epsilon/4)d(n,F)(n/r)^{k-1}$ copies of $F\in \mathcal{F}_e$, which contains a missing edge $e'\in M^\prime$ that intersects $e$.
Such a missing edge $e'$ can appear in at most $O_F(n^{f-k-2})$ copies of $F\in \mathcal{F}_e$ (note that $e'$ intersects $e$).
Hence by Lemma~\ref{counting F in Tp(n) with adding a matching}, we can conclude that there exists a vertex $v\in V(e)$ with
$$d_{M'}(v)\geq \frac{(\epsilon/8)d(n,F)(n/r)^{k-1}}{O_F(n^{f-k-2})}>\epsilon_{14} n.$$

Let $A=\{v\in V(H):d_{M^\prime}(v)>\epsilon_{14} n\}$.
We have argued above that every $e\in B_2$ has a vertex in $A$.
Consequently, since $|B^\ast|\geq |M^*|\geq |M^\prime|$, we have
$$2\sum_{v\in A} d_{B_2}(v)\geq 2|B_2|\geq 2\epsilon |B^\ast|\geq 2\epsilon |M^\prime|\geq \epsilon\sum_{v\in A} d_{M^\prime}(v)>\epsilon|A|\epsilon_{14} n.$$
By average, there exists $u\in A$ with $d_B(u)\geq d_{B^*}(u)\geq d_{B_2}(u)\geq \epsilon \epsilon_{14}n/2> \epsilon _{13} n$.
Without loss of generality, assume that $u\in V_1\setminus X$.
By the max-cut property of the partition $V_1\cup \ldots \cup V_r$,
the vertex $u$ has at least $\epsilon_{13}n$ neighbors in $V_i$ for each $i\in [r].$
Let $V^\prime_i$ be the set of common neighbors of $\{u\}\cup X$ in $V_i$ and let $Z=X\cup \{u\} \cup V^\prime_1\cup \ldots\cup V^\prime_r$.
We have $|V_i^\prime|\geq (\epsilon_{13}-(k-1)\epsilon_9)n \geq \epsilon_{12} n$.
For any $v\in V_1^\prime$, by Lemma~\ref{counting F in Tp(n) with adding a matching}, the number of the potential copies of $F$ in $H[Z]$ containing $uv$ is at least
$\left((\epsilon_{12}n-k)\right)^{k-1}\cdot d(r \epsilon_{12} n,F)>\epsilon_{11} n^{f-k-1}$.
Summing over all such $v\in V^\prime_1$, we obtain at least
$\epsilon_{12}n\cdot \epsilon_{11}  n^{f-k-1}\geq \epsilon_{10} n^{f-k}$ potential copies of $F$ containing $u$.
At least half of these potential copies of $F$ must have a missing edge $e'\in M^*$, as otherwise we get a contradiction to (\ref{low bound for H}).
By the definition of $Z$, every such $e'$ is not incident to $\{u\}\cup X$, thus $e'\in M^\prime$ and it appears at most $O_F(n^{f-k-2})$ copies of potential copies of $F$ containing $u$.
By double counting, a contradiction can be derived as follows
$$
\epsilon_{9}n^2< \frac{(\epsilon_{10}/2)n^{f-k}}{O_F(n^{f-k-2})}\leq |M^\prime|\leq |B^*|\leq \epsilon_4 n^2.
$$
This final contradiction finishes the proof of Claim~\rom{2}.
\end{proof}

Denote by $f(n,F)$ the minimum number of copies of $F$ obtained from $I_{k-1}+ T_r(n-k+1)$ by adding an edge (say $e$) to one class of $T_r(n-k+1)$
and removing all edges between $V(e)$ and $I_{k-1}$.\footnote{Note that here the edges between $V(e)$ and $I_{k-1}$ are deleted,
so there is a unique way of embedding $F$ in the resulting graph, i.e., first finding a $k$-matching consisting of $e$ and edges $x_iy_i$ for $1\leq i\leq k-1$ where $y_1,\ldots,y_{k-1}$ are from the same partite set and then embedding $F$ in the same way as in the definition of $d(n,F)$ (see the paragraph before Lemma~\ref{counting F in Tp(n) with adding a matching}).}
So we see $f(n,F)=(n/r)^{k-1}d(n,F)+O_F(n^{f-k-2})$ is a polynomial of degree $f-k-1$.

\medskip

\noindent {\bf Claim~\rom{3}.} Let $\omega=h(n,r,k)-e(H\setminus B^\ast).$ Then $|B^*|=q+\omega$, $|M^*|\leq \omega\leq \epsilon_5 n$,
and there exists an absolute positive constant $c=c(F)$ such that for each $e\in B^*$, $F(e)\geq f(n,F)- c\cdot \omega\cdot n^{f-k-2}$.

\begin{proof}
By the definition of $\omega$, we have $|B^\ast|= e(H)-e(H\setminus B^\ast)=e(H)-h(n,r,k)+\omega=q+\omega$.
So by Claim~\rom{2}, $\omega\leq |B^*|\leq \epsilon_5 n$.
Since $|B^*|-|M^*|=e(H)-e(H^*)=h(n,r,k)+q-e(H^*)\geq q$,
we see that $\omega=|B^\ast| -q \geq |M^\ast|$.
If $\omega=0$, then $M^*=\emptyset$ and $|B^*|=q$,
implying that $H(n,r,k)=H^*\subseteq H$.
In this case, the conclusion holds trivially.
So we may assume that $\omega\geq 1$.

Since $e(H\setminus B^\ast)=h(n,r,k)-\omega$,
it is easy to see that $e((H\setminus B^\ast)\setminus X)=t_r(n-k+1)-\omega+t$, where $t$ denotes the number of missing edges incident to $X$.
Note that $(H\setminus B^\ast)\setminus X$ is an $(n-k+1)$-vertex $r$-partite graph with the partition $(V_1\setminus X)\cup \ldots \cup (V_r\setminus X)$.
So by Lemma~\ref{Mubayi}, we have for each $i\in [r]$
\begin{equation*}\label{Sizes of V_i}
\left\lfloor\frac{n-k+1}{r}\right\rfloor-\omega+t\leq |V_i\setminus X|\leq \left\lceil\frac{n-k+1}{r}\right\rceil+\omega-t.
\end{equation*}
Consider an arbitrary edge $e\in B^\ast$. Without loss of generality, say $e\in H[V_1\setminus X]$.
Using the above bound on $|V_i\setminus X|$, the number of $k$-matchings each consisting of $e$ and $k-1$ edges in $H[X,V_1\setminus X]$ is at least $(|V_1\setminus X|-t-k)^{k-1}\geq \big(n/r-\omega-2k\big)^{k-1}$.
So the number of potential copies of $F$ (i.e., edges are only from $\{e\}\cup E(H^*)$),
each of which contains $e\cup X$ and has no edges between $V(e)$ and $X$, is at least $\big(n/r-\omega-2k\big)^{k-1}d(n-k-r-r\omega,F)=(n/r)^{k-1}d(n,F)+O_F(\omega)\cdot n^{f-k-2}$.
Such a potential copy of $F$ possibly contains some missing edge $e'\in M^*$, but every such $e'$ must have at least one endpoint outside of $V(e)\cup X$.
Since every such $e'$ lies in at most $O_F(n^{f-k-2})$ potential copies of $F$ counted above, we can derive that for some $c=c_F>0$,
\begin{align*}
F(e)\geq (n/r)^{k-1}d(n,F)+O_F(\omega)\cdot n^{f-k-2}- |M^*|\cdot O_F(n^{f-k-2})\geq f(n,F)- c\cdot \omega\cdot n^{f-k-2},
\end{align*}
where we use $|M^*|\leq \omega$ and $f(n,F)=(n/r)^{k-1}d(n,F)+O_F(n^{f-k-2})$. This proves Claim~\rom{3}.
\end{proof}

\section{Admissible color-$k$-critical graphs}\label{Section:proof of theorem 1.6}
In this section, we first introduce an ample subfamily of color-$k$-critical graphs (called {\bf admissible}; see Definition~\ref{def 3}),
which include all color-$\ell$-critical graphs for $\ell\in \{1,2\}$ and Kneser graphs $K(n,2)$.
Subsequently, we demonstrate that for any graph $F$ within this subfamily, there exists a constant $\delta>0$ such that the equality $h_F(n,q)=t_F(n,q)$ holds for all sufficiently large $n$ and all $1\leq q\leq \delta n$ (see Theorem~\ref{Thm:adm}).
These results collectively lead to the proof of Theorem~\ref{Thm:Knes}.

\subsection{Definitions and examples}
We now define the subfamily of color-$k$-critical graphs as mentioned above.
We begin by the following.

\begin{definition}\label{def 2}
\emph{Let $F$ be a graph with $\chi(F)=r+1\geq 3$. Let $\mathcal{F}=(F_0, F_1,\ldots,F_r)$ be an ordered sequence of graphs.\footnote{These graphs $F_i$ for $i\in \{0,1,...,r\}$ may be empty.}
Write $E(\mathcal{F})=\bigcup_{i=0}^r E(F_i)$.
If the graph $F_0+F_1+\ldots+F_r$ contains a copy of $F$ as its spanning subgraph and this $F$ contains all edges in $E(\mathcal{F})$,
then we say $\mathcal{F}$ is an {\it embedding type} (or for short, a {\it type}) of $F$.
Moreover, we let $\mathcal{F}_\alpha:=F_0$ be the {\it top} of the type $\mathcal{F}$ and $\mathcal{F}_\beta:=\bigcup_{i=1}^{r}F_i$ be the {\it bottom} of the type $\mathcal{F}$.
If $|V(F_0)|=\ell$, then we also call $\mathcal{F}$ an $\ell$-type.}
\end{definition}

The definition of types offers us a useful perspective for counting the number of copies $F$ in a graph with a given partition of $r+1$ parts.
For a graph $G$, let $\nu(G)$ be its {\it matching number}, i.e., the maximum size of a matching in $G$.

\begin{definition}\label{def 3}
\emph{A color-$k$-critical graph $F$ with $\chi(F)=r+1\geq 3$ is called {\bf admissible}, if
for any embedding type $\mathcal{F}=(F_0, F_1,\ldots,F_r)$ of $F$, the following hold that
\begin{itemize}
\item [(A).] $\nu(\bigcup_{i=1}^{r}F_i)\geq k-|V(F_0)|$, and
\item [(B).] if there is an edge in $F_0$, then $\nu(\bigcup_{i=1}^{r}F_i)\geq k+1-|V(F_0)|$.
\end{itemize}
}
\end{definition}

The family of admissible color-$k$-critical graphs forms a diverse and abundant collection.
In what follows, we will provide some notable examples and properties that showcase the richness of this family.
\begin{itemize}
\item All color-$\ell$-critical graphs $F$ for $\ell\in \{1,2\}$ are admissible.
The case when $\ell=1$ is trivial as both properties (A) and (B) are automatically satisfied.
Now we consider the case when $\ell=2$. First, the property (A) follows by the definition that $F$ is color-2-critical.
For (B), clearly it holds when $|V(F_0)|\geq 3$. So we may assume $|V(F_0)|\leq 2$.
Since there is an edge in $F_0$, we may assume $F_0$ is just an edge $ab$.
We need to show $\nu(\bigcup_{i=1}^{r}F_i)\geq 1$, which again follows by the definition.

\item In the coming subsection, we show that all Kneser graphs $K(n,2)$ belong to admissible color-$k$-critical graphs for $k=3$.

\item {\bf Proposition.} If $F_1$ is an admissible color-$k$-critical graph and $F_2$ is an admissible color-$\ell$-critical graph with $\chi(F_1)=\chi(F_2)$, then $F_1\cup F_2$ is an admissible color-$(k+\ell)$-critical graph.\\
    \vskip -3mm
    Repeatedly using this proposition, we see that the disjoint union of cliques of the same size (or more generally, the disjoint union of color-$\ell_i$-critical graphs $F_i$, where $\ell_i\in \{1,2\}$ for $i\in [t]$, of the same chromatic number) is an admissible color-$k$-critical graph for $k=\sum_{i\in [t]} \ell_i$.

\end{itemize}

\subsection{Kneser graphs}\label{Subs:Kneser}

Let $n,t$ be positive integers with $n\geq 2t+1$.
The {\it Kneser graph} $K(n,t)$ is the graph with the vertex set $\binom{[n]}{t}$,
where any two vertices $A,B\in \binom{[n]}{t}$ are adjacent if and only if $A\cap B=\emptyset$.
Answering a famous conjecture of Kneser \cite{Kneser}, Lov\'asz \cite{Lov78} proved that the chromatic number of $K(n,t)$ equals $n-2t+2$.
For a permutation $\pi$ on $[n]$, we say a $t$-subset of $[n]$ is {\it $\pi$-stable} if it contains no pairs $\{\pi(i),\pi(i+1)\}$ with $1\leq i<n$ nor the pair $\{\pi(1),\pi(n)\}$.
Schrijver \cite{Sch78} proved that for any permutation $\pi$ on $[n]$,
the induced subgraph of $K(n,t)$ on the vertex set consisting of all $\pi$-stable $t$-subsets of $[n]$ has the same chromatic number $n-2t+2$.\footnote{In fact, Schrijver \cite{Sch78} also proved that such an induced subgraph of $K(n,t)$ is {\it vertex-critical}, i.e., deleting any vertex will decrease the chromatic number.}

To the best of our knowledge, the cases $n\geq 6$ of the following lemma appear to be previously unestablished.

\begin{lemma}\label{Lem:K(n,2)}
For any $n\geq 5$, the Kneser graph $K(n,2)$ is color-$3$-critical with chromatic number $n-2$.
\end{lemma}

\begin{proof}
It is well known that $\chi(K(n,2))=n-2$.
First we claim that there exist three suitable edges in $K(n,2)$ whose removal will decrease the chromatic number to $n-3$.
For each $5\leq j\leq n$, let $V_j$ be the set consisting of all 2-sets $\{i,j\}$ with $1\leq i<j$,
and let $V_4$ be the set consisting of all 2-sets in $[4]$.
Then $V(K(n,2))=\bigcup_{j=4}^n V_j$, where $V_5,..., V_n$ form $n-4$ independent sets and $V_4$ induces a matching of size three in $K(n,2)$.
This proves the claim.

Next we show that deleting any two vertices from $K(n,2)$ will remain the same chromatic number $n-2$.
Let us consider any two vertices $A,B$ in $K(n,2)$.
Without loss of generality, we may assume $A=\{1,2\}, B=\{2,3\}$ or $A=\{1,2\}, B=\{3,4\}$.
In either case, it is easy to see that there exists a permutation $\pi$ on $[n]$ such that both $A$ and $B$ are not $\pi$-stable (e.g., taking $\pi(i)=i$ for each $i\in [n]$).
Then $K(n,2)-\{A,B\}$ contains all $\pi$-stable $2$-subsets of $[n]$ and thus by Schrijver's result \cite{Sch78}, it has chromatic number $n-2$.
Putting everything together, we see that $K(n,2)$ is color-$3$-critical.
\end{proof}

Combined with Theorem~\ref{extremal graph for mk-critical graph},
this shows that for $K=K(t,2)$, we have $\ex(n,K)=e(H(n,t-3,3)).$
The following lemma is the main result of this subsection.

\begin{lemma}\label{Lem2:K(n,2)}
For any $n\geq 5$, the Kneser graph $K(n,2)$ is an admissible color-$3$-critical graph.
\end{lemma}

\begin{proof}
By Lemma~\ref{Lem:K(n,2)}, we know that $K:=K(n,2)$ is color-$3$-critical.
It remains to show that $K$ is admissible.

Let $X$ be any critical subset in $K$, i.e., $|X|=3$ and $\chi(K-X)=2$.
First we claim that $X$ is an independent set in $K$.
For $n=5$, $K$ is just the Petersen graph and this is evident to see.
So we may assume $n\geq 6$.
Let $X=\{A, B, C\}$ and suppose for a contradiction that $A\cap B=\emptyset$.
Without loss of generality, we let $A=\{1,2\}$ and $B=\{3,4\}$.
There are four cases for $C$: namely, $(|C\cap A|, |C\cap B|)$ can be $(0,0), (0,1), (1,0)$ or $(1,1)$.
In each case, we can find a permutation $\pi$ on $[n]$ such that each of $A,B,C$ can be expressed as $\{\pi(i), \pi(i+1)\}$ for some $1\leq i<n$.
So $K-X$ contains all $\pi$-stable 2-subsets of $[n]$ and by Schrijver's result \cite{Sch78}, $\chi(K-X)=3$, a contradiction.
This shows the claim.

Consider any embedding type $\mathcal{F}=(F_0, F_1,\ldots,F_{n-3})$ of $K$.
All we need to show is that
\begin{itemize}
\item [(A).] $\nu(\bigcup_{i=1}^{n-3}F_i)\geq 3-|V(F_0)|$, and
\item [(B).] If there is an edge in $F_0$, then $\nu(\bigcup_{i=1}^{n-3}F_i)\geq 4-|V(F_0)|$.
\end{itemize}
For (A), there is nothing to prove if $|V(F_0)|\geq 3$.
If $|V(F_0)|=2$, as $K $ is color-$k$-critical, then we have $\chi(K-F_0)=\chi(K)=n-2$.
So there must be at least one edge in $\bigcup_{i=1}^{n-3}F_i$ (as otherwise $K-F_0$ can be partitioned into $n-3$ independent sets, a contradiction).
This shows that $\nu(\bigcup_{i=1}^{n-3}F_i)\geq 1$, i.e., (A) holds whenever $|V(F_0)|=2$.
If $|V(F_0)|=1$, then we claim that $\nu(\bigcup_{i=1}^{n-3}F_i)\geq 2$.
Otherwise, $\nu(\bigcup_{i=1}^{n-3}F_i)\leq 1$ and thus $\bigcup_{i=1}^{n-3}F_i$ has a vertex $x$ covering all its edges,
but this leads to that $\chi\left(K-\big(\{x\}\cup V(F_0)\big)\right)\leq n-3$, where $|\{x\}\cup V(F_0)|=2$, a contradiction to $K$ is color-$k$-critical.
Lastly, we consider $|V(F_0)|=0$.
In this case, we can derive $\nu(\bigcup_{i=1}^{n-3}F_i)\geq 3$ from that $K$ is color-$k$-critical easily.\footnote{This also can be derived from an exercise in the book of Matou\v{s}ek \cite{Mat03} (see Section 3.5, Exercise 3), which asserts that any coloring of $K(n,2)$ in $n-3$ colors contains at least three monochromatic edges.}

It remains to show (B). Suppose that $F_0$ contains an edge (so $|V(F_0)|\geq 2$).
If $|V(F_0)|\geq 4$, then again there is nothing to prove.
Suppose $|V(F_0)|=3$. If $\bigcup_{i=1}^{n-3}F_i$ contains no edges,
then $F_0$ becomes a critical subset in $F$ containing an edge, contradicting the above claim.
So $\bigcup_{i=1}^{n-3}F_i$ contains at least one edge, i.e., $\nu(\bigcup_{i=1}^{n-3}F_i)\geq 1$, as desired.
Lastly, we consider $|V(F_0)|=2$, i.e., $F_0$ is an edge say $ab$.
We want to show $\nu(\bigcup_{i=1}^{n-3}F_i)\geq 2$ in this case.
Suppose not. Then $\nu(\bigcup_{i=1}^{n-3}F_i)\leq 1$ and so $\bigcup_{i=1}^{n-3}F_i$ contains a vertex $c$ covering all its edges.
In this case, we observe that $\{a,b,c\}$ becomes a critical subset in $F$ which is not an independent set, again a contradiction to the above claim.
This proves that $K$ is admissible, completing the proof of this lemma.
\end{proof}

\subsection{Supersaturation for admissible graphs}
In the remainder, we present a proof of the main result of this section as follows.
This, in conjunction with Lemma~\ref{Lem2:K(n,2)}, provides a complete proof for Theorem~\ref{Thm:Knes}.

\begin{theorem}\label{Thm:adm}
For any admissible color-$k$-critical graph $F$, there exists a constant $\delta>0$ such that for any sufficiently large integer $n$ and any integer $1\leq q\leq \delta n$, we have
$h_F(n,q)=t_F(n,q).$
\end{theorem}

To prove this, we need two preliminary lemmas.
The following lemma helps us to bound the number of $\mathcal{F}$-types of admissible color-$k$-critical graphs $F$.

\begin{lemma}\label{bounded the numer of graph with given size}
Let $G$ be a graph with $m$ edges and $F$ be an $f$-vertex graph with minimum degree at lease one.
Then the number of copies of $F$ in $G$ is at most $O_F (m^{f-\nu(F)})$.
\end{lemma}
\begin{proof}
Fix a maximum matching $\{e_1,...,e_\nu\}$ of $F$ where $\nu=\nu(F)$.
Then it is not hard to see that there exists a spanning {\it star-forest} $F^\prime$ (i.e., a forest consisting of stars) in $F$ such that it has $\nu$ stars and each star contains exactly one edge $e_i$ for $i\in [\nu]$.
Then $|V(F^\prime)|=f$, $e(F^\prime)=f-\nu(F)$ and the number of copies of $F^\prime$ in $G$ is at most ${m \choose f-\nu(F)}$. Since a copy of $F^\prime$ in $G$ is contained in at most $2^{O_F(1)}$ copies of $F$ in $G$ and each copy of $F$ in $G$ contains a copy of $F^\prime$ in $G$, the result follows.
\end{proof}

Let $F$ be a color-$k$-critical graph with $\chi(F)=r+1$.
Let $\ell\geq 0$ be an integer and $\mathcal{F}=(F_0, F_1,\ldots,F_r)$ be an $\ell$-type of $F$.
Fix disjoint sets $V_i$ of size $n_i\geq |V(F_i)|$ for $i\in [r]$, where $n/2\leq \sum_{i\in [r]} n_i\leq n$.
Let $K_\mathcal{F}$ be obtained from $\overline{K}_{\ell}+K(V_1,\ldots,V_r)$ by embedding $E(F_0)$ into $\overline{K}_{\ell}$ and $E(F_i)$ into $V_i$ for $i\in [r]$.
Denote by $c_{\mathcal{F}}(n_1,\ldots,n_r)$ the number of copies of $F$ in $K_\mathcal{F}$ containing all edges of $E(\mathcal{F})$.

The following lemma can be easily proven by the same argument as Lemmas~\ref{counting F in H(n,p,k) with adding an edge} and~\ref{estimate c(n,f)}, the details of which are omitted here.
Let $i(G)$ be the number of isolated vertices of a graph $G$.

\begin{lemma}\label{the number copies in l-type}
Let $\mathcal{F}$ be an $\ell$-type of $F$ and $n$ be sufficiently large.
Let $\sum_{i=1}^r n_i=\sum_{i=1}^rn^\prime_i\in [n/2,n]$ where $\max_{i,j} |n^\prime_i-n^\prime_j|\leq 1$.
Define $a_i=n_i-n^\prime_i$ for $i\in[r]$ and $ A=\max\{|a_i|: i\in\{1,\ldots,r\} \}$.
Then there exists a constant $\eta_\mathcal{F}>0$ such that (recall $\mathcal{F}_\beta$ denotes the bottom of the type $\mathcal{F}$)
$$|c_{\mathcal{F}}(n_1,\ldots,n_r)-c_{\mathcal{F}}(n^\prime_1,\ldots,n^\prime_r)|\leq \eta_\mathcal{F}\cdot A \cdot n^{i(\mathcal{F}_\beta)-1},$$
where $c_{\mathcal{F}}(n_1,\ldots,n_r)$ is a multi-polynomial of degree $i(\mathcal{F}_\beta)$.
\end{lemma}

We are ready for the proof of Theorem~\ref{Thm:adm}.

\medskip

\noindent{\bf Proof of Theorem~\ref{Thm:adm}.}
Fix an admissible color-$k$-critical graph $F$ with $\chi(F)=r+1$ and $f=|V(F)|$.
Let $1/n\ll \delta\ll \epsilon_1\ll \epsilon_2\ll ... \ll \epsilon\ll 1$ be sufficiently small so that Claims~\rom{1}, \rom{2} and \rom{3} in Subsection~\ref{subsec:refined} hold.
Let $H$ be an $n$-vertex graph on $h(n,r,k)+q$ edges with minimum number of copies of $F$, where $1\leq q\leq \delta n$.
Then we can partition $V(H)=X\cup V_1\cup \ldots \cup V_r$ such that $|X|=k-1$ and the following hold.
Let $M$ be the set of non-edges of $H$ between $X, V_1,\ldots, V_r$, and let $B_i=E(H[V_i])$.\footnote{Using the terminologies $M^*$ and $B^*$ from Subsection~\ref{subsec:refined} (see the paragraph before Claim~\rom{2}), here we have $M=M^*\backslash E[X]$ and $\bigcup_{i\in [r]} B_i=B^*$, where $E[X]$ consists of all edges with both vertices in $X$.}
Let $m=|M|$, $b_i=|B_i|$, $b=\sum_{i=1}^{r}b_i$, and $\omega=b-q$ be from Claim~\rom{3}.
Then $\epsilon n\geq b=q+\omega\geq q+m$.

To show the equality $h_F(n,q)=t_F(n,q)$, it suffices to prove that $H$ contains $H(n,r,k)$ as a subgraph.
We gradually achieve this.
Initially, we establish a crucial inequality as indicated in \eqref{remove an edge and add an edge}.
An edge $uv\in \bigcup_{i\in [r]} B_i$ is called {\it bad}.
Denote by $\#F(uv)$ the number of copies of $F$ of $H$ containing $uv$ as the unique bad edge.
Then $\#F(uv)=\Omega_F(n^{f-k-1})$ by Claim~\rom{3}.
For any $xy\in M$, define
$$\# F^\prime(xy)=\#F(H+xy)- \#F(H)$$
to be the number of {\it transitional} copies of $F$ associated with $uv$, that is, the number of copies of $F$ generated by including the non-edge $xy$ of $H$.
Now we assert that
\begin{equation}\label{remove an edge and add an edge}
\# F^\prime (xy)=\Omega_F(n^{f-k-1}) \mbox{ holds for all $xy\in M$}.
 \end{equation}
To see this, we point out that in fact, $\# F^\prime (xy)\geq \# F(uv)$ for any $xy\in M$ and bad edge $uv$
(as, otherwise, we can reduce the number of copies of $F$ by deleting $uv$ and adding $xy$, a contradiction).

For any copy of $F$ contained in $H$, it corresponds to a unique type $\mathcal{F}=(F_0,F_1,...,F_r)$, namely,
where $F_0$ denotes the induced subgraph of this $F$ within $X$ and, for $i\in [r]$, $F_i$ denotes the induced subgraph of this $F$ within $V_i$.
Let $n_i=|V_i|$ for $i\in [r]$.
We can bound the number of copies of $F$ in $H$ from above by summarizing the number of edge-sets $E(\mathcal{F})$ multiplying $c_{\mathcal{F}}(n_1,\ldots,n_r)$ over all types $\mathcal{F}$.
Utilizing this counting strategy, we will now proceed to demonstrate the following two claims.

For integers $\ell\geq 0$, let $Y_\ell$ be the collection of all $\ell$-types of $F$ and $\widehat{Y}_\ell \subseteq Y_\ell$ be the collection of $\ell$-types of $F$ whose top contains at least one edge.

\medskip

\noindent{\bf Claim \rom{4}.} There exist $k-1$ vertices of $H$ with degree $n-1$.

\begin{proof}
By Claims \rom{1} and \rom{2}, each vertex in $X=\{x_1,\ldots,x_{k-1}\}$ has degree more than $n-\epsilon n$ and $|B|\leq \epsilon n$.
We first show that $H[X]$ is a complete graph.
Suppose that, without loss of generality, there exist $x,y\in X$ with $xy\notin E(H[X])$.
Our goal is to obtain an upper bound of $\#F^\prime(xy)$ that contradicts \eqref{remove an edge and add an edge}.
Note that for any copy of $F$ in $H+xy$ which contains $xy$,
it corresponds to the following unique type $\mathcal{F}=(F_0,F_1,...,F_r)$, where $F_0$ denotes the induced subgraph of this $F$ on $V(F)\cap X$ and for $i\in [r]$,
$F_i$ denotes the induced subgraph of this $F$ on $V(F)\cap V_i$.
Let $\ell=|V(F_0)|$.
Since $xy\in E(F_0)$, the type $\mathcal{F}$ is from $\widehat{Y}_\ell$ for some $2\leq \ell\leq k-1$.
Since $F$ is admissible, we have $\nu(\mathcal{F}_\beta)\geq k-\ell+1$.
Let $\mathcal{N}$ denote the number of edge-sets $E(\mathcal{F}_\beta)=\bigcup_{i\in [r]} E(F_i)$ in $H-X$ where each $E(F_i)\subseteq E(H[V_i])$.
By Lemma~\ref{bounded the numer of graph with given size}, we have
$$\mathcal{N}\leq O_F(|B|^{|V(\mathcal{F}_\beta)|-\nu(\mathcal{F}_\beta)-i(\mathcal{F}_\beta)})\leq O_F(\epsilon)\cdot n^{(f-\ell)-(k-\ell+1)-i(\mathcal{F}_\beta)}\leq O_F\left(\epsilon\cdot n^{f-k-1-i(\mathcal{F}_\beta)}\right).$$
The number of edge-sets $E(\mathcal{F}_\alpha)=E(F_0)$ in $H[X]+xy$ containing $xy$ is $O_F(1)$.
Hence, we have
\begin{align*}
\#F^\prime(xy)\leq \sum_{\ell=2}^{k-1} \sum_{\mathcal{F}\in \widehat{Y}_\ell} O_F(1)\cdot \mathcal{N} \cdot c_{\mathcal{F}}(n_1,...,n_r)\leq O_F\big(\epsilon\cdot n^{f-k-1}\big),
\end{align*}
where the first inequality follows by treating the edges between $V_1,...,V_r$ and $V(F_0)$ are complete for any given edge-set $E(\mathcal{F})$,
and the last inequality holds because of Lemma~\ref{the number copies in l-type} that $c_{\mathcal{F}}(n_1,...,n_r)$ is a multi-polynomial of degree $i(\mathcal{F}_\beta)$. This is a contradiction to \eqref{remove an edge and add an edge}.
Thus, $H[X]$ is a complete graph.

Now we show that $x_i$ is adjacent to each vertex of $V_j$ in $H$ for $i\in [k-1]$ and $j\in [r]$.
Suppose for a contradiction that there is a vertex $y \in V_1$ such that $x_1y\notin E(H)$.
For any copy of $F$ in $H+x_1y$ containing $x_1y$,
it corresponds to a unique type $\mathcal{F}=(F_0,F_1,...,F_r)\in \widehat{Y}_\ell$ with some $2\leq \ell\leq k$,
where $F_0$ denotes the induced subgraph of this copy $F$ on $V(F)\cap (X\cup \{y\})$ and
for $i\in [r]$, $F_i$ denotes the induced subgraph of this $F$ on $V(F)\cap (V_i-\{y\})$.
Since $F$ is admissible, we have $\nu(\mathcal{F}_\beta)\geq k-\ell+1$.
Let $\widehat{V_i}=V_i-\{y\}$ for $i\in [r]$.
Slightly modifying the above argument, we can obtain
\begin{eqnarray*}
\#F^\prime(x_1y)&\leq& \sum_{\ell=2}^{k} \sum_{\mathcal{F}\in \widehat{Y}_\ell} \left(O_F(\epsilon)\cdot n^{f-k-1-i(\mathcal{F}_\beta)}\right) \cdot c_{\mathcal{F}}\big(|\widehat{V_1}|,...,|\widehat{V_r}|\big)\leq O_F\big(\epsilon\cdot n^{f-k-1}\big),
\end{eqnarray*}
a contradiction to (\ref{remove an edge and add an edge}).
The proof of Claim \rom{4} is complete.\end{proof}

\noindent{\bf Claim \rom{5}.} We have $M=\emptyset$.

\begin{proof}
Suppose that $M\neq \emptyset$, say $uv\in M$.
For any copy of $F$ in $H+uv$ containing $uv$,
it corresponds to a unique type $\mathcal{F}=(F_0,F_1,...,F_r)\in \widehat{Y}_\ell$ with some $2\leq \ell\leq k+1$.
Here, $F_0$ denotes the induced subgraph of this copy $F$ on $V(F)\cap (X\cup \{u,v\})$ and
for $i\in [r]$, $F_i$ denotes the induced subgraph of this $F$ on $V(F)\cap (V_i-\{u,v\})$.
Since $F$ is admissible, we have $\nu(\mathcal{F}_\beta)\geq k-\ell+1$.
Let $\widehat{V_i}=V_i-\{u,v\}$ for $i\in [r]$.
Then similarly, we have
$\#F^\prime(uv)\leq \sum_{\ell=2}^{k+1} \sum_{\mathcal{F}\in \widehat{Y}_\ell} \left(O_F(\epsilon)\cdot n^{f-k-1-i(\mathcal{F}_\beta)}\right) \cdot c_{\mathcal{F}}\big(|\widehat{V_1}|,...,|\widehat{V_r}|\big)\leq O_F\big(\epsilon\cdot n^{f-k-1}\big),$
again contradicting (\ref{remove an edge and add an edge}).
This proves Claim~\rom{5}.   \end{proof}

By Claims~\rom{4} and \rom{5}, we see that $K_{k-1}+K(V_1,\ldots,V_r)\subseteq H$.
Recall that $\epsilon n\geq \omega=b-q\geq 0$.
To complete the proof of Theorem~\ref{Thm:adm}, it suffices to show that either $\omega=0$ or $A:=\max_{i,j}\left||V_i|-|V_j|\right|\leq 1$
(if so, then we have $H(n,r,k)\subseteq H$, as desired).
Suppose for a contradiction that $\omega\geq 1$ and $A\geq 2$.
By Lemma~\ref{Mubayi}, we have $A=O(\omega)$.
Fix $B_q$ to be a subset of $\bigcup_{i\in [r]} B_i$ consisting of $q$ edges.
Let $H^{\prime}$ be obtained from $H(n,r,k)$ by adding all edges of $B_q$.
Note that $H^{\prime}$ has the same number of edges as $H$.
Since $F$ is admissible, for any $\mathcal{F}\in Y_\ell$, we have $\nu(\mathcal{F}_\beta)\geq k-\ell$.
Then the number of copies of $F$ in $H^\prime$ satisfies that (for $i\in [r]$ let $n_i'\in \{\lceil (n-k+1)/r \rceil,\lfloor(n-k+1)/r \rfloor\}$ such that $\sum_{i\in [r]} n_i'=n-k+1$)
$$\#F(H') \leq \sum_{\ell=0}^{k-1}\sum_{\mathcal{F}\in Y_\ell}O_F\big(|B_q|^{|V(\mathcal{F}_\beta)|-\nu(\mathcal{F}_\beta)-i(\mathcal{F}_\beta)}\big)\cdot c_\mathcal{F}(n'_1,\ldots,n'_r)  \leq O_F(\epsilon\cdot n^{f-k}).$$
Let $T$ be the number of copies of $F$ in $H$ only using bad edges from $B_q$.
By Lemma~\ref{the number copies in l-type}, we see that
$$|T-\#F(H')|\leq \#F(H')\cdot O_F(A/n)\leq O_F\big(\epsilon\cdot A n^{f-k-1}\big).$$
There are $\omega=b-q$ edges in $\big(\bigcup_{i\in [r]} B_i\big)\backslash B_q$ not used in any copy of $F$ in $H$ contributed to $T$,
hence
\begin{align*}
\#F(H)&\geq T+\omega\cdot c(n_1,...,n_r;F)\\
&\geq \left(\#F(H')-O_F\big(\epsilon\cdot A n^{f-k-1}\big)\right)+\omega\cdot \left(c(n,F)-O_F(A) n^{f-k-2} -O_F(A^2)n^{f-k-3}\right)\\
&\geq \#F(H')-O_F\big(\epsilon\cdot \omega\cdot n^{f-k-1}\big)+\Omega_F(\omega\cdot n^{f-k-1})> \#F(H'),
\end{align*}
where the second inequality holds by Lemma~\ref{estimate c(n,f)} and the third inequality holds because $c(n,F)$ is a polynomial of degree $f-k-1$ and $A=O(\omega)=O(\epsilon n)$.
This contradicts the minimality of $\#F(H)$ and thus completes the proof of Theorem~\ref{Thm:adm}.
\QED

\section{Proof of Theorem~\ref{Thm:hVSt}}\label{section s and t}
The goal of this section is to prove Theorem~\ref{Thm:hVSt}.
To explain and describe the intricate thresholds of Theorem~\ref{Thm:hVSt},
we need to get deeper into the structure of a color-$k$-critical graph $F$.
In the rest of this section, we always assume that $k\geq 2, r\geq 2$ and $F$ denotes a color-$k$-critical graph with $\chi(F)=r+1$.

We begin by introducing some new parameters on $F$.
Let $\lambda(F)$ denotes the minimum size of a subset $A\subseteq V(F)$ satisfying $\chi(F\setminus A)=r$.
Let $\mathbb{X}(F)=\{A\subseteq V(F): |A|=\lambda(F) \mbox{ and } \chi(F\setminus A)=r\}$ be the family of all {\it critical subsets} of $F$.
For a critical subset $A\in \mathbb{X}(F)$, let $\mathcal{V}(A)$ denote the family of all possible partitions $\{U_1,\ldots,U_r\}$ of $V(F\setminus A)$ such that each $U_i$ is stable.
For $A\in \mathbb{X}(F)$ and any integer $\ell\geq 1$, if there exist $x\in A$ and $U_j\in \{U_1,\ldots,U_r\} \in \mathcal{V}(A)$ with $|N_F(x)\cap U_j|\geq \ell$, then let
$$\delta_\ell(A)=\min\{|N_F(x)\cap U_j|:x\in A,~U_j\in \{U_1,\ldots,U_r\}\in \mathcal{V}(A) \mbox{~and~} |N_F(x)\cap U_j|\geq \ell\};$$
otherwise, let $\delta_\ell(A)=\infty$.
We now define two parameters playing crucial roles in this section.
Let $$t(F)=\min_{\mbox{$A\in \mathbb{X}(F): A$ is stable}}\delta_2(A) \mbox{ ~~ and ~~ } s(F)=\min_{\mbox{$A\in \mathbb{X}(F): A$ is not stable}}\delta_1(A).\footnote{Here if none of $A\in \mathbb{X}(F)$ satisfies the requirement, then the corresponding parameter is defined to be $\infty$.}$$
For example, if $F$ consists of $k$ vertex-disjoint copies of $K_{r+1}$, then $t(F)=\infty$ and $s(F)=\infty$.

Now we are able to state the main result of this section, which implies Theorem~\ref{Thm:hVSt}.

\begin{theorem}\label{Theorem-main3}
Let $F$ be given as in the first paragraph of this section with additional properties that $t(F)\in [4,\infty)$ and $s(F)\geq 2$.\footnote{We can prove similar results for the case $t=3$ (in this case, the extremal graphs for the supersaturation problem may be obtained by putting either a triangle or a star into each part of $H(n,r,k)$; see Lemma~\ref{lem:for contain M2} for some hints), but it requires more effort, so we have decided not to pursue it. We would like to treat the case $t=\infty$ in a forthcoming paper.}
Then there exists $\epsilon>0$ such that the following hold for sufficiently large $n$ and
any $n$-vertex graph $H$ on $h(n,r,k)+q$ edges with minimum number of copies of $F$:
\begin{itemize}
\item [(a)] if $1\leq q\leq \epsilon n^{1-1/s(F)}$, then $H$ contains $H(n,r,k)$ as a subgraph, and
\item [(b)] if $n^{1-1/s(F)}/\epsilon\leq q\leq \epsilon n$, then $H$ does not contain $H(n,r,k)$ as a subgraph.
\end{itemize}
\end{theorem}

Using Theorem~\ref{Theorem-main3}, one can derive Theorem~\ref{Thm:hVSt} promptly in the following.

\begin{proof}[\bf Proof of Theorem~\ref{Thm:hVSt} (assuming Theorem~\ref{Theorem-main3}).]
The case $s=1$ follows by Theorem~\ref{main-example} easily.
Consider $s\geq 2$.
By Theorem~\ref{Theorem-main3} and Lemma~\ref{lem:color-k-critical-is-stable}, it suffices to construct non-bipartite color-$k$-critical graphs $F$ with $t(F)\in [4,\infty)$ and $s(F)=s$.
Such graphs exist as illustrated in Figure~1(b),
where the present graph $F$ is obtained from a copy of $M_k+k\cdot S_{s+1}$ by adding exactly one edge between two centers of these $S_{s+1}$'s (for some $k\geq 3$).
One can verify that such $F$ is color-$k$-critical with $t(F)=k+s-1\geq 4$ and $s(F)=s$.
\end{proof}

In what follows, we prove Theorem~\ref{Theorem-main3} by first establishing some useful properties.
For the proof, we need to consider some special $n$-vertex graphs and use them to derive upper bounds on $h_F(n,q)$.
Fix $1\leq q\leq \epsilon n$ for some small real $\epsilon>0$. Let $L=\{\ell_1, \ldots ,\ell_r\}$ be a set of non-negative integers with $\sum_{i=1}^{r}\ell_i=q$.
Denote by $H(L)$ the graph obtained from $H(n,r,k)$ by adding $r$ stars with $\ell_1,\ldots,\ell_r$ edges into the $r$ parts of $H(n,r,k)$ respectively.
Let $X$ be the vertex set of the clique $K_{k-1}$ in $H(n,r,k)$ and let $C$ be the set of centers of these embedded stars in $H(L)$.\footnote{Note that we also use $X$ to denote the set of the $k-1$ vertices given by Claim~\rom{1} of Section~\ref{sec_properties-of-H}. This may cause confusion at the first sight, but we would like to use $X$ at both circumstances as they refer to the same set of vertices conceptually. If a star is a single edge, then one can choose any one of its vertices as its center.}
Denote by $H^\prime(L)$ the graph obtained from $H(L)$ by deleting all edges inside $X\cup C$.
So $e(H^\prime(L))=e(H(L))-{k-1+\alpha_L \choose 2}$, where $\alpha_L$ denotes the number of positive integers in $L$.

The following propositions are useful for estimating the copies of $F$ in the proof of Theorem~\ref{Theorem-main3}.

\begin{proposition}\label{ob 1}
Each copy of $F$ in $H^\prime(L)$ or $H(L)$ contains at least $k$ vertices in $X\cup C$.
Moreover, if a copy of $F$ in $H(L)$ contains exactly $k$ vertices and at least one edge in $X\cup C$,
then inside this copy $F$, every $x\in V(F)\cap C$ is incident to at least $s(F)$ edges of the embedded star with center $x$.
\end{proposition}
\begin{proof}
Note that after deleting $X\cup C$, both $H^\prime(L)$ and $H(L)$ have chromatic number $r$.
So if a copy of $F$ in $H^\prime(L)$ or $H(L)$ contains at most $k-1$ vertices in $X\cup C$, then we get a contradiction to that $F$ is color-$k$-critical.
If a copy of $F$ in $H(L)$ contains exactly $k$ vertices and at least one edge in $X\cup C$,
then these $k$ vertices form a non-stable set $A\in \mathbb{X}(F)$.
Let $V_1,\ldots,V_r$ denote the $r$ parts of $H(n,r,k)\setminus X$.
It is clear that any vertex $x\in A$ has at least one neighbor in each $V_i\cap V(F\setminus A)$.
By definition of $s(F)$, we see that any $x\in V(F)\cap C\subseteq A$ is incident to at least $s(F)$ edges in its own $V_i$, which must be from the embedded star.
The proof is complete.
\end{proof}

\begin{proposition}\label{ob 3}
Let $t:=t(F)<\infty$ and $s:=s(F)\geq 2$. Then the following hold that
$$\mathcal{N}_F(H^\prime(L))=q\cdot c(n,F)+ \sum_{i=1}^{r} \beta(\ell_i)n^{f-k-t} +\sum_{i\neq j}O(\ell_i\ell_jn^{f-k-2}),$$
where $\beta(x)=ax^t+\Theta(x^{t-1})$ for some absolute constant $a>0$, and
$$\mathcal{N}_F(H(L))=\mathcal{N}_F(H^\prime(L))+ \sum_{i=1}^{r} \Theta (\ell_i^s n^{f-k-s})+\sum_{i\neq j}O(\ell_i\ell_jn^{f-k-3}).$$
\end{proposition}
\begin{proof}
Each copy of $F$ in $H^\prime(L)$ contains either (1) exactly one edge from the embedded stars, (2) at least two edges from some embedded star and no edges from other stars, or (3) some edges from at least two distinct embedded stars.
For (1), there are exactly $q c(n,F)$ such copies of $F$ (note that as $s(F)\geq 2$, these deleted edges in $E(H(L))\setminus E(H^\prime(L))$ will not affect on this count).
For (2), such a copy of $F$ has a stable set $A\in \mathbb{X}(F)$ consisting of the $k-1$ vertices in $X$ and the center $y$ of the involved star,
and by definition of $t(F)$, in this copy $F$ the vertex $y$ is incident to at least $t=t(F)$ edges of the star.
So there are $\sum_{i=1}^{r} \beta(\ell_i)n^{f-k-t}$ many copies of $F$ of the type $(2)$, where $\beta(x)=ax^t+\Theta(x^{t-1})$ for some $a>0$.
For (3), the number of such copies of $F$ is $\sum_{i\neq j}O(\ell_i\ell_jn^{f-k-2})$ (note that possibly there exists no copy of $F$ of this kind).
Putting all together, we now can derive the equation on $\mathcal{N}_F(H^\prime(L))$.

The estimation on $\mathcal{N}_F(H(L))-\mathcal{N}_F(H^\prime(L))$ can be proved similarly, by classifying among (2) and (3) under the additional condition that such copies of $F$ use at least one edge in $X\cup C$.
By Proposition~\ref{ob 1}, the number of copies of $F$ of the type $(2)$ which also use at least one edge in $X\cup C$ is $\sum_{i=1}^{r} \Theta (\ell_i^s n^{f-k-s})$.
Since $s\geq 2$, the number of copies of $F$ in the type $(3)$ is $\sum_{i\neq j} \Theta (\ell_i^s\ell_j^s n^{f-k-2s})+\sum_{i\neq j}O(\ell_i\ell_jn^{f-k-3})$ (where the second term is the number of copies of $F$ using at least $k+1$ vertices in $X\cup C$).
Using the above arguments, it is easy to derive the equation on $\mathcal{N}_F(H(L))$.
\end{proof}
The following lemma is technical.
In its simplest case $|I|=1$, it provides a lower bound on a linear combination of the number of matchings of size two and the number of stars (say with $t$ edges). %in graphs with given order and size.

\begin{lemma}\label{lem:for contain M2}
Fix a real $\alpha>0$, integers $t\geq 3$, $r\geq 2$ and a non-empty set $I$ of indexes with $|I|\leq r$.
Let $1\gg\delta\gg\epsilon\gg 1/n>0$ be sufficiently small compared to $\alpha, t$ and $r$.
For every $i\in I$, let $G_i$ be a graph with $m_i\leq \epsilon n$ edges such that if $t=3$ and $m_i=3$, then $G_i$ is not a triangle.\footnote{It is easy to see that this lemma does not hold if $t=3$ and all $G_i$'s are triangles.}
If there exists an index $j\in I$ with $\Delta(G_j)\leq (1- \delta)m_j$,
then
$$
\alpha\cdot\sum_{i\in I}  \mathcal{N}_{M_2}(G_i)n^{t|I|-2}  \geq \prod_{i\in I}{m_i \choose t}.
$$
If $\Delta(G_i)\geq (1- 3\delta)m_i$ for every $i\in I$, then
$$
\alpha\cdot\sum_{i\in I}  \mathcal{N}_{M_2}(G_i)n^{t|I|-2} + \prod_{i\in I}{\Delta(G_i)\choose t}\geq \prod_{i\in I}{m_i \choose t}.
$$
\end{lemma}
\begin{proof}
We may assume that $m_i\geq t$ for all $i\in I$ (as otherwise ${m_i \choose t}=0$ and it holds trivially).
If there exists some $i\in I$ with $\Delta(G_i)\leq m_i/2$, then $\mathcal{N}_{M_2}(G_i)\geq {m_i \choose 2}- \sum_\ell  {d_\ell\choose 2}\geq{m_i \choose 2}- 2{m_i/2 \choose 2}\geq m_i^2/5$, where $\{d_\ell\}_\ell$ denotes the degree sequence of $G_i$.
Since $0<\epsilon\ll \alpha$ and $m_i\leq \epsilon n$,  we  have
$$\alpha \cdot \mathcal{N}_{M_2}(G_i)n^{t|I|-2}\geq(\alpha/5)\cdot m_i^2 n^{t|I|-2}\geq  \prod_{i\in I}{m_i \choose t},$$
from which all of the conclusions hold.
Now suppose $\Delta(G_i)> m_i/2$ for every $i\in I$.
Then we claim
\begin{equation}\label{equ:N(M2)}
\mathcal{N}_{M_2}(G_i)\geq (m_i-\Delta(G_i))\cdot m_i/5.
\end{equation}
We note $m_i\geq t\geq 3$.
If $m_i=3$, then $t=3$ and thus $G_i$ is not a triangle. In this case, it is easy to check that \eqref{equ:N(M2)} holds.
So we may assume $m_i\geq 4$.
Then $\mathcal{N}_{M_2}(G_i)\geq (m_i-\Delta(G_i))\cdot (\Delta(G_i)-2)\geq (m_i-\Delta(G_i))\cdot m_i/5,$
where the first inequality holds by a simple fact and the last inequality follows from $\Delta(G_i)> m_i/2$ and $m_i\geq 4$.

If there exists some $j\in I$ with $\Delta(G_j)\leq (1- \delta)m_j$,
then by \eqref{equ:N(M2)} we have $\mathcal{N}_{M_2}(G_j)\geq \delta m_j^2/5$.
Using $n\geq m_j/\epsilon$ and $\delta\gg \epsilon$, we have
\begin{align*}
\alpha \cdot \mathcal{N}_{M_2}(G_j)n^{t|I|-2}\geq (\alpha\delta m_j^2/5)\cdot n^{t|I|-2}\geq \frac{\alpha \delta m_j^t}{5\epsilon^{t-2}}\cdot \prod_{i\neq j}{m_i\choose t }\geq \prod_{i\in I}{m_i \choose t},
\end{align*}
which implies the desired first conclusion.
Now we assume that $\Delta(G_i)\geq (1-3\delta)m_i$ for every $i\in I$.
As $(m_i-\Delta(G_i))/m_i\leq 3\delta$ is sufficiently small,
we can get $$\binom{\Delta(G_i)}{t}\geq \binom{m_i}{t}-O_t(1)\cdot (m_i-\Delta(G_i))\cdot m_i^{t-1}.$$
Using the previous inequality and \eqref{equ:N(M2)}, we can further get that (note that $\epsilon\ll \alpha, t$)
\begin{align*}
\alpha \cdot \mathcal{N}_{M_2}(G_i)n^{t-2}+\binom{\Delta(G_i)}{t}&\geq (\alpha/5\epsilon^{t-2})\cdot(m_i-\Delta(G_i))\cdot m_i^{t-1}+{\Delta(G_i) \choose t }\\
&\geq \binom{m_i}{t}+ (\alpha/6\epsilon^{t-2})\cdot (m_i-\Delta(G_i))\cdot m_i^{t-1}.
\end{align*}
Let $\ell\in I$ be the index which maximizes $(m_\ell-\Delta(G_\ell))\cdot m_\ell^{t-1}/\binom{m_\ell}{t}$.
Then we can derive that
\begin{align*}
&\alpha\sum_{i\in I}\mathcal{N}_{M_2}(G_i)n^{t|I|-2} + \prod_{i\in I}\binom{\Delta(G_i)}{t}\\
\geq &\alpha \cdot \mathcal{N}_{M_2}(G_\ell)n^{t|I|-2} +\binom{\Delta(G_\ell)}{t}\cdot \prod_{i\in I\backslash\{\ell\}}{\Delta(G_i)\choose t}
\geq \bigg(\alpha \cdot \mathcal{N}_{M_2}(G_\ell)n^{t-2}+{\Delta(G_\ell)\choose t}\bigg)\cdot \prod_{i\in I\backslash\{\ell\}}{\Delta(G_i)\choose t}\\
\geq &\bigg(\binom{m_\ell}{t}+ (\alpha/6\epsilon^{t-2})\cdot(m_\ell-\Delta(G_\ell))\cdot m_\ell^{t-1}\bigg)\cdot \prod_{i\in I\backslash\{\ell\}}\bigg(\binom{m_i}{t}-O_t(1)\cdot (m_i-\Delta(G_i))\cdot m_i^{t-1}\bigg)\geq \prod_{i\in I}{m_i \choose t},
\end{align*}
where the last inequality holds by the choice of $\ell$ and the fact that $\alpha/6\epsilon^{t-2}\gg O_t(1)$.
The proof of this lemma is complete.
\end{proof}

We are ready to present the proof of Theorem~\ref{Theorem-main3}.

\medskip

{\noindent \bf Proof of Theorem~\ref{Theorem-main3}.}
Fix $k\geq 2, r\geq 2$ and a color-$k$-critical graph $F$ with $\chi(F)=r+1$ such that $t:=t(F)\geq 4$ and $s:=s(F)\geq 2$.
Let $1\gg\delta\gg \epsilon\gg 1/n>0$ be sufficiently small to satisfy claims of Section~\ref{sec_properties-of-H} and Lemma~\ref{lem:for contain M2} (where $1\gg\delta\gg \epsilon>0$ are from Lemma~\ref{lem:for contain M2} and the constant $\alpha_{\ref{lem:for contain M2}}$ there will be determined later).
Let $H$ be an $n$-vertex graph on $h(n,r,k)+q$ edges with minimum number of copies of $F$, where $1\leq q\leq \epsilon n$.
Then using Claims~\rom{1}, \rom{2} and \rom{3} in Subsection~\ref{subsec:refined}, the following hold.
One can partition $V(H)=X\cup V_1\cup \ldots \cup V_r$, where $|X|=k-1$ and each vertex in $X$ has degree at least $n-\epsilon n$.
Let $M$ be the set of missing edges of $H$ between $X, V_1,\ldots, V_r$, and let $B_i=E(H[V_i])$.
Let $m=|M|$, $b_i=|B_i|$, $b=\sum_{i=1}^{r}b_i$, and $\omega=b-q$ be from Claim~\rom{3}.
Then $\epsilon n\geq b=q+\omega\geq q+m$.

We divide the proof into two parts depending on the range of $q$.
Throughout this proof, we use $B^\ast=\bigcup_{i=1}^r B_i$ and $L_\ell=\{\ell,0,\ldots,0\}$.
For a set of edges $A$, we denote by $\Delta(A)$ the maximum degree of the graph induced by the edges in $A$.

\medskip

{\bf Case (A).} $1\leq q \leq \epsilon n^{1-1/s}$.

\medskip

In this case our goal is to show $H(n,r,k)\subseteq H$.
Recall the definition of $f(n,F)$ before Claim~\rom{3},
and note that $c(n,F)=f(n,F)$ whenever $s\geq 2$.

We first prove the case $s=2$.
This proof is straightforward and reveals the main proof idea, that is, to construct a ``well-designed'' graph with the same numbers of vertices and edges but with less copies of $F$ than $H$.
Let $a:=\max\{q,\omega\}\leq \epsilon n$.
By Claim~\rom{3}, if $\omega\geq 1$, then
$\mathcal{N}_F(H)\geq \sum_{e\in B}F(e)\geq (q+\omega)(c(n,F)- \Theta(\omega) n^{f-k-2})\geq qc(n,F)+\omega c(n,F)-a\Theta(\omega) n^{f-k-2}= qc(n,F)+\Omega(n^{f-k-1})$.
By Proposition~\ref{ob 3}, we have $\mathcal{N}_F(H(L_q))=qc(n,F)+O(q^2)n^{f-k-2}$.
Since $q\leq \epsilon n^{1/2}$, we derive that $\mathcal{N}_F(H)\geq qc(n,F)+\Omega(n^{f-k-1})> \mathcal{N}_F(H(L_q))$,
a contradiction to the minimality of $\mathcal{N}_F(H)$.
Thus $\omega=0$, from which we can derive $H(n,r,k)\subseteq H$ (i.e., see the proof of Claim~\rom{3}), as desired.

From now on we consider the general case $s\geq 3$.
Our proof strategy is (again) to show that whenever $\omega\geq 1$, one can construct an $n$-vertex graph with $h(n,r,k)+q$ edges whose number of copies of $F$ is strictly smaller than $\mathcal{N}_F(H)$.
For that, we need to estimate the number of copies of $F$ more precisely and thus we introduce several notations in the following paragraphs.
First, denote by $f(M_2)$ the minimum number of copies of $F$ obtained from $I_{k-1}+ T_r(n-k+1)$ by adding a copy of $M_2$ to one class of $T_r(n-k+1)$ and removing all edges between $V(M_2)$ and $I_{k-1}$.
Let $i\geq 1, j\geq 0$ be integers with $i+j\leq r$.
Denote by $f_{i,j}(S_{t+1})$ the minimum number of copies of $F$ obtained from $I_{k-1}+ T_r(n-k+1)$ by adding a copy of $S_{t+1}$ to each of $i$ classes of $T_r(n-k+1)$,
adding an edge into each of other $j$ classes of $T_r(n-k+1)$, and removing all edges inside $C^\prime\cup I_{k-1}$,
where $C^\prime$ is the set of centers of embedding stars $S_{t+1}$ and edges.\footnote{Here, we view both vertices of an embedding edge as its centers.}
Denote by $f^\ast(M_\ell)$ the minimum number of copies of $F$ obtained from $I_{k-1}+ T_r(n-k+1)$ by adding one edge to each of $\ell$ classes of $T_r(n-k+1)$ and removing all edges inside $C^{\prime\prime}\cup I_{k-1}$, where $C^{\prime\prime}$ consists of vertices of these $\ell$ edges.
It is not hard to see that there exist reals $\alpha>0, \beta_{i,j}\geq 0$, and $\gamma_\ell\geq 0$ satisfying $f(M_2)=\alpha n^{f-k-2}+O(n^{f-k-3})$, $f_{i,j}(S_{t+1})=\beta_{i,j} n^{f-k-ti-j}+O(n^{f-k-ti-j-1})$ (for $t= \infty$, let $f_{i,j}(S_{t+1})=0$), and $f^\ast(M_\ell)=\gamma_\ell n^{f-k-\ell}+O(n^{f-k-\ell-1})$.

For comparison, we name analogous types of $F$ in $H$.
Let $i\in [r]$. For a copy of $M_2$ in $B_i$, let $F(M_2)$ be the number of copies of $F$ in $H$ containing this $M_2$ as the only edges from $B^\ast$.
For disjoint $I, J\subseteq [r]$, denote by $F_{I,J}(S_{t+1})$ the number of copies of $F$ in $H$ containing a copy of $S_{t+1}$,
whose leaves are completely adjacent to $X$ in $H$, in every $B_{i}$ for $i\in I$ and containing an edge in every $B_{j}$ for $j\in J$.\footnote{Here it is important to require that the leaves of stars $S_{t+1}$ are complete to $X$, for the validation of \eqref{eq: for counting F 2}. Also note that because of $s(F)\geq 3$, every copy of $F$ counted in $F_{I,J}(S_{t+1})$ cannot contain any edge between $X$ and the unique edge in $B_{j}$ for every $j\in J$.}
For a copy of $M_\ell$ with at most one edge in each $B_i$, denote by $F^\ast(M_\ell)$ the number of copies of $F$ in $H$ containing this $M_\ell$ as the only edges in $B^\ast$.
Repeating the proof of Claim~\rom{3}, one can similarly obtain the following estimations: there exists $c>0$ such that
\begin{equation}\label{eq: for counting F 1}
  F(M_2)\geq f(M_2) - c\cdot \omega\cdot n^{f-k-3},
\end{equation}
\begin{equation}\label{eq: for counting F 2}
F_{I,J}(S_{t+1})\geq  f_{|I|,|J|}(S_{t+1})- c\cdot \omega\cdot n^{f-k-t|I|-|J|-1}, \mbox{where $|I|\geq 1$,}
\end{equation}
\begin{equation}\label{eq: for counting F 3}
F^\ast(M_\ell)\geq f^\ast(M_\ell) - c\cdot \omega\cdot n^{f-k-\ell-1}.
\end{equation}

Recall $1\gg\delta\gg \epsilon>0$ from Lemma~\ref{lem:for contain M2} which are defined in the beginning of this proof.
We claim that there exist subsets $B_i^\ast\subseteq B_i$ for each $i\in [r]$ such that
\begin{itemize}
\item[(A1).] $\sum_{i=1}^{r}b^\ast_i \geq b-m\geq q$ (where $b^\ast_i:=|B^\ast_i|$), and
\item[(A2).] either $\Delta(B^\ast_i)<(1-\delta)b^\ast_i$, or there exists a vertex $u$ with $\Delta(B^\ast_i)\geq d_{B^\ast_i}(u)\geq (1-3\delta)b^\ast_i$ and every vertex in $N_{B^\ast_i}(u)$ is completely adjacent to $X$.\footnote{The latter case is consistent with the definition of $F_{I,J}(S_{t+1})$.}
\end{itemize}
To see this, we run the following algorithm (within $H$) for each $i\in [r]$.
Initially, let $B^\ast_i=B_i$ and $b^\ast_i= |B^\ast_i|$.
If $\Delta(B^\ast_i)\geq (1-\delta)b^\ast_i$ and some vertex $v$ of degree one in $B^\ast_i$ is incident to $X$ by a missing edge of $H$,
then we delete the unique edge of $v$ from $B^\ast_i$;
repeat the above process until we cannot delete any edges.
When this process ends, either $\Delta(B^\ast_i)<(1-\delta)b^\ast_i$, or $\Delta(B^\ast_i)\geq (1-\delta)b^\ast_i$.
In the latter case, let $u$ be the vertex with $d_{B^\ast_i}(u)=\Delta(B^\ast_i)$, and all vertices of degree one in $B^\ast_i$ are complete to $X$.
If there is some $v\in N_{B^\ast_i}(u)$ incident to $X$ by a missing edge of $H$ (which must have degree two in $B^\ast_i$ and there are at most $2\delta b^\ast_i$ such vertices), then we delete $uv$ from $B^\ast_i$.
In the end, we see that $d_{B^\ast_i}(u)\geq (1-3\delta)b^\ast_i$ and every vertex in $N_{B^\ast_i}(u)$ is completely adjacent to $X$.
Since the number of the deleted edges is at most $m$ the number of missing edges in $H$,
we have $\sum_{i=1}^{r}b^\ast_i \geq b-m\geq q$.

Hence there exist non-negative integers $\ell_1, \ldots, \ell_{r}$ with $\sum_{i=1}^{r}\ell_i=q$ and $\ell_i\leq b^\ast_i\leq b_i$.
Let $L=\{\ell_1, \ldots, \ell_{r}\}$. Note that $H(L)$ is an $n$-vertex graph with $h(n,r,k)+q$ edges.

In the remaining of the proof, we compare $\mathcal{N}_F(H)$ with $\mathcal{N}_F(H(L))$.
First we consider $H(L)$.
Let $1\leq \ell\leq k$. The number of copies of $F$ containing edges of $\ell$ embedding stars (no edges from other stars) and at least $k-\ell+1$ vertices of $X$ in $H(L)$ is $O(q^\ell n^{f-k-\ell-1})=o(n^{f-k-1})$.
Now we bound the number of copies of $F$ containing edges from $\ell$ embedding stars (no edges from other stars) and exactly $k-\ell$ vertices of $X$ in $H(L)$; call them {\it standard}.
If a standard copy of $F$ contains an edge from $X\cup C$, then by Proposition~\ref{ob 1}, it must contain $s$ edges from some embedding star,
so the number of such standard copies is $O(q^s n^{f-k-s})=O(\epsilon^s)\cdot n^{f-k-1}$, where we use $q\leq \epsilon n^{1-1/s}$.
It remains to consider standard copies of $F$ containing none of the edges from $X\cup C$; call them {\it feasible}.
Let
\begin{itemize}
\item  $\mathcal{W}_\ell=\{$copies of $F$ in $H(L)$ containing exactly $\ell$ independent edges from the embedding stars of $H(L)$ and containing exactly $k-\ell$ vertices of $X\}$, where $1\leq \ell\leq k$,\footnote{We point out that each copy $F$ in $\mathcal{W}_\ell$ for $\ell\geq 1$ is feasible by definition (as $s(F)\geq 2$).}
\item  $\mathcal{R}_{I,J}=\{$feasible copies of $F$ in $H(L)$ containing a copy of $S_{t+1}$
     in each of the embedding stars of sizes $\ell_i$ for $i\in I$,
containing an edge in each of the embedding stars of sizes $\ell_j$ for $j\in J$, and containing exactly $k-|I|-|J|$ vertices of $X\}$, where $I,J\subseteq [r],I\cap J=\emptyset$ and $|I|\geq 1$.
\end{itemize}
By the definition of $t(F)$, each feasible copy of $F$ belongs to either $\bigcup_{\ell=1}^k \mathcal{W}_\ell$ or $\mathcal{R}_{I,J}$ for some $I,J\subseteq [r], I\cap J=\emptyset$ and $|I|\geq 1$.
Putting these all together, we have the following estimation on $\mathcal{N}_F(H(L))$
\begin{equation}\label{eq-for-compare}
q\cdot c(n,F) +\sum_{\ell=2}^{k}|\mathcal{W}_\ell|+\sum_{I,J\subseteq [r],I\cap J=\emptyset, |I|\geq 1}|\mathcal{R}_{I,J}|\geq \mathcal{N}_F(H(L))-O(\epsilon^s)\cdot n^{f-k-1}.
\end{equation}

For the purpose of comparison, we consider the following pairwise disjoint collections of copies of $F$ in $H$ (again by no mean of a partition;
recall the sets $B^\ast_{i}$ from the properties (A1) and (A2)):
\begin{itemize}
\item  $\mathcal{W}^\ast_\ell=\{$copies of $F$ containing exactly $\ell$ independent edges in $\ell$ parts $V_i$'s as the only edges of $\bigcup_{\alpha\in [r]} B_\alpha^\ast$ and containing exactly $k-\ell$ vertices of $X\}$, where $1\leq \ell\leq k$,
\item  $\mathcal{R}^\ast_{I,J}=\{$copies of $F$ containing a copy of $S_{t+1}$ %(\textcolor{red}{copies of $F$ containing a copy of $K_3$ when $B^\ast_i$ induces a triangle})
    in every $B^\ast_{i}$ for $i\in I$ and an edge in every $B^\ast_{j}$ for $j\in J$ as the only edges in $\bigcup_{\alpha\in [r]} B_\alpha^\ast$ and containing exactly $k-|I|-|J|$ vertices of $X\}$, where $I,J\subseteq [r],I\cap J=\emptyset$ and $|I|\geq 1$,
\item  $\mathcal{T}^\ast_i=\{$copies of $F$ containing a copy of $M_2$ in $B^\ast_i$ as the only edges in $\bigcup_{\alpha\in [r]} B_\alpha^\ast$, where $i\in [r]$.
\end{itemize}
Clearly, we have
\begin{equation}\label{eq-for-compare2}
\mathcal{N}_F(H)\geq \sum_{e\in B^\ast}F(e)+\sum_{\ell=2}^{k}|\mathcal{W}^\ast_\ell|+\sum_{I,J\subseteq [r],I\cap J=\emptyset, |I|\geq 1}|\mathcal{R}^\ast_{I,J}|+\sum_{i=1}^{r}|\mathcal{T}^\ast_i|.
\end{equation}
In the following, we will show that assuming $\omega\geq 1$,
\begin{equation}\label{eq add 1}
|\mathcal{W}^\ast_\ell|\geq |\mathcal{W}_\ell|- O(q^\ell)\omega n^{f-k-\ell-1} \mbox{ for each $2\leq \ell\leq k$ }
\end{equation}
and
\begin{equation}\label{eq add 2}
\sum_{I,J\subseteq [r],I\cap J=\emptyset, |I|\geq 1}|\mathcal{R}^\ast_{I,J}|+\sum_{i=1}^{r}|\mathcal{T}^\ast_i|\geq \sum_{I,J\subseteq [r],I\cap J=\emptyset, |I|\geq 1}|\mathcal{R}_{I,J}|-\Theta(\omega \epsilon n^{f-k-1}).
\end{equation}

Let us first show that to complete the proof for Case (A), it suffices to show \eqref{eq add 1} and \eqref{eq add 2}.
Indeed, by combining \eqref{eq-for-compare2}, \eqref{eq add 1}, \eqref{eq add 2} with Claim~\rom{3} of Section~\ref{sec_properties-of-H}, assuming $\omega\geq 1$ we have
\begin{align*}
\mathcal{N}_F(H)\geq & \sum_{e\in B^\ast}F(e)+ \sum_{\ell=2}^{k} |\mathcal{W}^\ast_\ell| +\sum_{I,J\subseteq [r],I\cap J=\emptyset, |I|\geq 1}|\mathcal{R}^\ast_{I,J}|+\sum_{i=1}^{r}|\mathcal{T}^\ast_i|\\
\geq &(q+\omega)\left(c(n,F)-\Theta (\omega n^{f-k-2})\right)+\sum_{\ell=2}^{k} \left(|\mathcal{W}_\ell|-O(q^\ell\omega n^{f-k-\ell-1})\right)\\
&+\sum_{I,J\subseteq [r],I\cap J=\emptyset, |I|\geq 1}|\mathcal{R}_{I,J}|-\Theta( \omega \epsilon n^{f-k-1})\\
\geq  & qc(n,F)+\sum_{\ell=2}^{k} |\mathcal{W}_\ell|+\sum_{I,J\subseteq [r],I\cap J=\emptyset, |I|\geq 1}|\mathcal{R}_{I,J}|+\omega\cdot \Theta(n^{f-k-1}),
\end{align*}
where the last inequality follows from that $q\leq \epsilon n$ and $c(n,F)$ is a polynomial of degree $f-k-1$ with variable $n$.
If $\omega\geq 1$, then we can derive the following contradiction that
$$\mathcal{N}_F(H)\geq qc(n,F)+\sum_{\ell=2}^{k} |\mathcal{W}_\ell|+\sum_{I,J\subseteq [r],I\cap J=\emptyset, |I|\geq 1}|\mathcal{R}_{I,J}|+ \Theta(n^{f-k-1})> \mathcal{N}_F(H(L)),$$
where the last equality follows by \eqref{eq-for-compare}.
Hence $\omega=0$, which implies that $H$ contains $H(n,r,k)$ as a subgraph, thus proving Case (A).

Turning back to \eqref{eq add 1} and \eqref{eq add 2}, we will first prove \eqref{eq add 1}.
Since $\sum_{i=1}^{r}(b_i-\ell_i)=b-q=\omega\geq 1$, there is an integer $\beta$ with $b_{\beta}-\ell_{\beta}\geq \omega/r$.
Let $\ell'_i=\ell_i$ for each $i\in [r]\backslash \{\beta\}$ and $\ell'_\beta=\ell_\beta+\omega/r$ so that $b_i\geq \ell_i'$ for all $i\in [r]$.
Fix $2\leq j\leq k$.
Recall the definition of $f^\ast(M_j)$, which equals $\gamma_j n^{f-k-j}+O(n^{f-k-j-1})$ for some $\gamma_j\geq 0$.
Since $s(F)=s\geq 3$, we see that all copies of $F$ in $\mathcal{W}_j$ (in $H(L)$) are contributed in the same way as counted in $f^\ast(M_j)$.
So $|\mathcal{W}_j|\leq \sum_{k_1,\ldots,k_j\subseteq  [r]} \ell_{k_1}\ldots \ell_{k_{j}}\big(\gamma_j n^{f-k-j}+O(n^{f-k-j-1})\big)$.
If $\gamma_j=0$, then \eqref{eq add 1} holds trivially.
So assume $\gamma_j>0$.
By \eqref{eq: for counting F 3}, we see that $|\mathcal{W}^\ast_j|-|\mathcal{W}_j|$ equals
\begin{align*}
&\sum_{k_1,\ldots,k_j\subseteq  [r]} b_{k_1}\ldots b_{k_{j}}\left(\gamma_j n^{f-k-j}-c \omega n^{f-k-j-1}\right)-\sum_{k_1,\ldots,k_j\subseteq  [r]} \ell_{k_1}\ldots \ell_{k_j}\big(\gamma_j n^{f-k-j}+O(n^{f-k-j-1})\big)\\
\geq &\sum_{k_1,\ldots,k_j\subseteq  [r]} \ell'_{k_1}\ldots \ell'_{k_{j}}\left(\gamma_j n^{f-k-j}-c \omega n^{f-k-j-1}\right)-\sum_{k_1,\ldots,k_j\subseteq  [r]} \ell_{k_1}\ldots \ell_{k_{j}}\gamma_j n^{f-k-j}-O(q^jn^{f-k-j-1})\\
\geq &\omega/r\cdot \left(\sum_{k_1,\ldots,k_{j-1}\subseteq [r]\setminus \{\beta\}}\ell_{k_1}\ldots \ell_{k_{j-1}}\right)\cdot \gamma_j n^{f-k -j}\\
&-\left(\ell'_\beta\sum_{k_1,\ldots,k_{j-1}\subseteq  [r]\setminus \{\beta\}} \ell_{k_1}\ldots \ell_{k_{j-1}}+\sum_{k_1,\ldots,k_j\subseteq  [r]\setminus \{\beta\}} \ell_{k_1}\ldots \ell_{k_{j}}\right)c \omega n^{f-k-j-1}-O(q^jn^{f-k-j-1})\\
\geq &- O(q^j)\omega n^{f-k-j-1},
\end{align*}
where the last inequality holds because $\ell'_\beta\leq b_\beta\leq b\leq \epsilon n$, $\ell_i\leq q\leq \epsilon n$ for any $i\in [r]\setminus \{\beta\}$, and $\epsilon$ is sufficiently small (i.e., $\epsilon\ll \gamma_j/c$).
This proves \eqref{eq add 1}.

Now we consider \eqref{eq add 2}.
Fix disjoint $I,J\subseteq[r]$ with $|I|\geq 1$.
Then $$|\mathcal{T}^\ast_{i}|=\mathcal{N}_{M_2}(B^\ast_{i})\cdot F(M_2) ~~\mbox{ and } ~~ |\mathcal{R}^\ast_{I,J}|\geq \left(\prod_{i\in I}{\Delta (B^\ast_{i}) \choose t} \prod_{j\in J}b^\ast_j\right) \cdot F_{I,J}(S_{t+1}),$$
where $F(M_2)\geq (\alpha-o(1))\cdot n^{f-k-2}$ and $F_{I,J}(S_{t+1})=(1-o(1))f_{|I|,|J|}(S_{t+1})=(\beta_{|I|,|J|}-o(1))\cdot n^{f-k-|I|t-|J|}$ from \eqref{eq: for counting F 1} and \eqref{eq: for counting F 2} respectively.
Since $t\geq 4$, all $B_i^*$ for $i\in I$ satisfy Lemma~\ref{lem:for contain M2}.
Using the property (A2) and Lemma~\ref{lem:for contain M2} with $\alpha_{\ref{lem:for contain M2}}=\alpha/(2\cdot 3^{r}\beta_{|I|,|J|}+1)$,
\begin{align*}
\frac{1}{3^{r}}\sum_{i\in I}|\mathcal{T}^\ast_{i}|+|\mathcal{R}^\ast_{I,J}|
\geq &\left(\alpha_{\ref{lem:for contain M2}}\cdot \sum_{i\in I}\mathcal{N}_{M_2}(B^\ast_{i}) n^{t|I|-2}+ \prod_{i\in I}{\Delta(B^\ast_{i})\choose t}\right)\cdot \left(\prod_{j\in J}b^\ast_j\cdot F_{I,J}(S_{t+1})\right)\\
\geq &\prod_{i\in I}{b^\ast_{i} \choose t}\cdot \left(\prod_{j\in J}b^\ast_j \cdot F_{I,J}(S_{t+1})\right).
\end{align*}
Using $|\mathcal{R}_{I,J}|=\left(\prod_{i\in I}{\ell_{i} \choose t}\prod_{j\in J}\ell_j\right) \cdot f_{|I|,|J|}(S_{t+1})$,
since $\ell_i\leq b_i^\ast$ and $\ell_i\leq q\leq \epsilon n$ for each $i$, we obtain
\begin{align}\label{equ:1/3r}
\frac{1}{3^{r}}\sum_{i\in I}|\mathcal{T}^\ast_{i}|+  |\mathcal{R}^\ast_{I,J}|-|\mathcal{R}_{I,J}|
\geq &\prod_{i\in I}{b^\ast_{i} \choose t}\cdot \prod_{j\in J}b^\ast_j \cdot F_{I,J}(S_{t+1})-\prod_{i\in I}{\ell_{i} \choose t}\cdot\prod_{j\in J}\ell_j \cdot f_{|I|,|J|}(S_{t+1})\nonumber
\\\nonumber
\geq &\left(\prod_{i\in I}{\ell_{i} \choose t}\prod_{j\in J}\ell_j\right) \cdot (F_{I,J}(S_{t+1})-f_{|I|,|J|}(S_{t+1}))
\geq -\Theta(\omega q^t n^{f-k-t-1}),
\end{align}
where the last inequality holds because of \eqref{eq: for counting F 2} and $|I|\geq 1$.
Summing up the above inequalities for all $I,J\subseteq [r]$ with $I\cap J=\emptyset$ and $|I|\geq 1$ (there are at most $3^r$ many such inequalities), we can easily derive \eqref{eq add 2}.
The proof of Case (A) is complete.

\medskip

{\bf Case (B).} $ n^{1-1/s} /\epsilon  \leq q\leq \epsilon n$.

\medskip

Suppose for a contradiction that $H$ contains $H(n,r,k)$ as a subgraph.
First consider the case $s=2$.
Without loss of generality, let $e(B_1)\geq q/r$.
Note that $\mathcal{N}_{M_2}(B_1)+\mathcal{N}_{S_3}(B_1)\geq{q/r \choose 2}$ and the number of copies of $F$ contains exactly two edges of $B_1$ (which are incident or not) is $\Theta(n^{f-k-2})$.
Let $q^\ast=q+\binom{k}{2}$.
Then since $\epsilon n\geq q\geq n^{1/2}/\epsilon$ and $t\geq 4$, we have
$\mathcal{N}_F(H)\geq q c(n,F)+\Theta(q^2)n^{f-k-2}\geq q^\ast c(n,F)+\Theta(q^t/\epsilon^2)n^{f-k-t}>\mathcal{N}_F(H^\prime(L_{q^\ast}))$, where the last inequality holds by the first inequality of Proposition~\ref{ob 3}.
This is a contradiction as $H^\prime(L_{q^\ast})$ has the same numbers of vertices and edges as $H$.

Assume that $s\geq 3$.
Let $L=\{\ell_1, \ldots, \ell_{r}\}$, where $\ell_i=b_i$ if $b_i\geq 4$ and  $\ell_i=0$ otherwise.
We first compare $\mathcal{N}_F(H)$ with $\mathcal{N}_F(H(L))$.
Fix $\alpha\in [r], I,J\subseteq [r],I\cap J=\emptyset$, $|I|\geq 1$ and $\tau\geq k-|I|-|J|$.
Let $\eta=\min\{s,t\}$ if $\tau=k-|I|-|J|$ and $\eta=2$ if $\tau\geq k-|I|-|J|+1$.
Define
\begin{itemize}
\item  $\mathcal{R}^\prime_{I,J,\tau}=\{$copies of $F$ in $H(L)$ containing a star of size at least $\eta$
     in each of the embedding stars of sizes $\ell_i$ for $i\in I$, containing an edge in each of the embedding stars of sizes $\ell_j$ for $j\in J$, and containing exactly $\tau$ vertices of $X\}$,
\item  $\mathcal{R}^\ast_{I,J,\tau}=\{$copies of $F$ in $H$ containing a star of size at least $\eta$ in $B_i$ for $i\in I$, containing an edge in each $B_j$ for $j\in J$, and containing exactly $\tau$ vertices of $X\}$,
\item  $\mathcal{T}^\ast_\alpha=\{$copies of $F$ in $H$ containing a copy of $M_2$ in $B_i$ as the only edges in $\bigcup_{\beta\in [r]} B_\beta$.
\end{itemize}
We note that for $\tau\geq k-|I|-|J|+1$, both $|\mathcal{R}^\prime_{I,J,\tau}|$ and $|\mathcal{R}^\ast_{I,J,\tau}|$ become lower order terms than when $\tau=k-|I|-|J|$.
Similar as the proof of (\ref{eq add 2}) in Case (A), by Lemma~\ref{lem:for contain M2},
we can show that
\begin{equation}\label{equ:CaseB}
\frac{1}{3^{r}k}\sum_{i\in I}|\mathcal{T}^\ast_{i}|+  |\mathcal{R}^\ast_{I,J,\tau}|\geq |\mathcal{R}^\prime_{I,J,\tau}|.
\end{equation}
Note that $H$ contains a copy of $H(n,r,k)$, so compared with the proof in Case (A), we have $\omega=0$ and thus the counting proof here is easier.
It is also clear that the number of copies of $F$ in $H$ containing exactly $\ell$ independent edges in $\ell$ parts $V_i$'s as the only edges of $\bigcup_{\alpha\in [r]} B_\alpha$ is larger than the number of copies of $F$ in $H(L)$ containing exactly $\ell$ independent edges in $\ell$ embedding stars as the only embedding edges.
Therefore, putting all copies of $F$ together (e.g., summing up \eqref{equ:CaseB} for all $I,J\subseteq [r]$ with $I\cap J=\emptyset$, $|I|\geq 1$ and $\tau$),
we can obtain $$\mathcal{N}_F(H)\geq \mathcal{N}_F(H(L)).$$
Hence, using the second inequality of Proposition~\ref{ob 3}, we get
$$\mathcal{N}_F(H)-\mathcal{N}_F(H^\prime(L))\geq \mathcal{N}_F(H(L))-\mathcal{N}_F(H^\prime(L))\geq \Theta(q^s)n^{f-k-s}.$$
Let $L^\ast=\{\ell_1+{k-1+\alpha_L \choose 2}+\sum_{i=1}^{r}(b_i-\ell_i),\ell_2\ldots,\ell_r\}$, where $\alpha_L$ is the number of positive integers in $L$.
Then $H^\prime(L^\ast)$ has the same number of edges as $H$, and using Proposition~\ref{ob 3} again,
we obtain that $\mathcal{N}_F(H^\prime(L^\ast))=\mathcal{N}_F(H^\prime(L))+\Theta(n^{f-k-1})$.
Finally putting the above all together, since $q\geq n^{1-1/s}/\epsilon$,
\begin{align*}
\mathcal{N}_F(H)&\geq \mathcal{N}_F(H^\prime(L))+\Theta(q^s)n^{f-k-s}\geq \mathcal{N}_F(H^\prime(L))+\Theta(n^{f-k-1}/\epsilon^s)>\mathcal{N}_F(H^\prime(L^\ast)),
\end{align*}
a contradiction. The proof of Theorem~\ref{Theorem-main3} is complete.
\QED

\section{Concluding remarks}
In this paper, we explore the supersaturation problem and present several results, both positive and negative, that extend beyond the existing framework. These findings offer new insights into the complexity and intricate nature of this problem for general graphs.
We now proceed to discuss some remarks and related problems.

Let $F$ be a color-$k$-critical graph with $\chi(F)=r+1$.
In Section~\ref{sec_properties-of-H}, we establish several general properties for supersaturated graphs of $F$ (that is, graphs of given order and size with the minimum number of copies of $F$). Using these properties, one can quickly prove a general lower bound on $h_F(n,q)$ as follows.
Recall the definition of $f(n,F)$ (from Claim~\rom{3} of Subsection~\ref{subsec:refined}), 
which denotes the minimum number of copies of $F$ obtained from $I_{k-1}+ T_r(n-k+1)$ by adding an edge (say $e$) to one class of $T_r(n-k+1)$ and removing all edges between $V(e)$ and $I_{k-1}$.

\begin{theorem}
Fix $k\geq 1$ and any color-$k$-critical graph $F$ with $\chi(F)=r+1$.
Then there exists a constant $\delta = \delta_F>0$ such that if $n$ is sufficiently large and $1\leq q\leq \delta n$, then
$h_F (n, q) \geq q\cdot f(n, F)$.
\end{theorem}
\begin{proof}
Let $H$ be an $n$-vertex graph with $\ex(n,F)+q$ edges and minimum number of copies of $F$, as stated in the beginning of Subsection~\ref{subsec:refined}.
Then we see that Claims~\rom{1}, \rom{2} and \rom{3} in Subsection~\ref{subsec:refined} hold.
Let $\omega$ be from Claim~\rom{3} and let $a=\max\{q,\omega\}$.
Then $a\leq \epsilon_5 n$. By Claim~\rom{3}, the number of copies of $F$ is at least
\begin{align*}
\sum_{e\in B^*}F(e)\geq (q+ \omega)(f(n,F)-  c\cdot \omega\cdot n^{f-k-2})
\geq qf(n,F)+\omega f(n,F)-2a \cdot c\cdot \omega\cdot n^{f-k-2} \geq qf(n,F),
\end{align*}
where the last inequality follows because $f(n,F)$ is a polynomial of degree $f-k-1$ and thus $\omega f(n,F)-2a \cdot c\cdot \omega\cdot n^{f-k-2}\geq 0$ for sufficiently large $n$. The proof is complete.
\end{proof}

This result can be seen as an extension of Theorem~\ref{Thm:Mubayi} since the notation $f(n,F)$ corresponds to $c(n,F)$ when $k=1$.

Let $F$ be a color-critical graph. As mentioned earlier, Pikhurko and Yilma \cite{pikhurko2017} asymptotically determined $h_F(n,q)$ in the range $q=o(n^2)$. Investigating the asymptotic behavior of $h_F(n,q)$ when $q=\Omega(n^2)$ suggests by itself an challenging problem.
A good starting point might be to examine the case when $F$ is an odd cycle.

In Theorem~\ref{main-example}, we show that Conjecture~\ref{con-mubayi} does not hold in the graph case.
As discussed after Theorem~\ref{main-example},
assuming $n$ is sufficiently large, there exist non-bipartite stable graphs $F$ such that $h_F(n,q)<t_F(n,q)$ holds for any fixed integer $q\geq 4$.
This leads us to inquire whether the same result holds for the cases $q\in \{1,2,3\}$.
In contrast to our findings for $q\geq 4$, we speculate that Conjecture~\ref{con-mubayi} holds in the intriguing case $q=1$.
Furthermore, we believe that the equality $h_F(n,1)=t_F(n,1)$ holds for the majority of graphs $F$, regardless of whether it is stable or bipartite.
Consequently, we pose the following question.
\begin{question}
Is it true that for any graph $F$ containing a cycle and for sufficiently large $n$, the equality $h_F(n,1)=t_F(n,1)$ holds?
\end{question}

Based on our current knowledge, all graphs for which the extremal graphs have been determined provide positive evidence for this question.
Nevertheless, it remains an interesting problem to pursue Conjecture~\ref{con-mubayi}
in the context of hypergraphs or graphs with chromatic number three.

\end{document}